\documentclass[12pt]{article}
\usepackage[margin=1in]{geometry}
\usepackage{amsmath, amsthm, amssymb, mathtools, bm,mathrsfs}
\usepackage{natbib}
\usepackage{bbm}
\usepackage{booktabs} 
\usepackage{enumitem}
\usepackage[colorlinks=true, allcolors=blue]{hyperref}
\usepackage{comment}
\usepackage{xcolor}
\usepackage{algorithm}
\usepackage[noend]{algpseudocode}
\title{Worst-Case Maximal Inequalities for Heavy-tailed Random Vectors}
\author{Woonyoung Chang\footnote{\href{mailto:woonyouc@gmail.com}{woonyouc@gmail.com}}}
\date{\today}

\newtheorem{theorem}{Theorem}[section]

\newtheorem{lemma}[theorem]{Lemma}
\newtheorem{proposition}[theorem]{Proposition}
\newtheorem{corollary}[theorem]{Corollary}
\theoremstyle{definition}

\newtheorem{remark}{Remark}

\DeclarePairedDelimiterX{\mnorm}[1]
{\vvvert}
{\vvvert}
{\ifblank{#1}{\:\cdot\:}{#1}}

\newcommand{\floor}[1]{\lfloor#1\rfloor}

\newcommand{\ceil}[1]{\lceil#1\rceil}

\newcommand{\norm}[1]{\Vert#1\Vert}
\newcommand{\Abs}[1]{\left|#1\right|}
\newcommand{\abs}[1]{|#1|}
\newcommand{\Set}[1]{\left\{#1\right\}}

\newcommand{\Real}{\mathbb R}

\newcommand{\Dc}{\mathcal{D}}
\newcommand{\Ec}{\mathcal{E}}

\newcommand{\Mc}{\mathcal{M}}

\newcommand{\Pc}{\mathcal{P}}

\newcommand{\Eb}{\mathbb{E}}

\newcommand{\Pb}{\mathbb{P}}

\newcommand{\Pp}{\mathbb P}
\newcommand{\one}{\mathbbm{1}}
\allowdisplaybreaks

\begin{document}

\maketitle
\begin{abstract}
This paper establishes finite-sample worst-case maximal inequalities for averages of independent centered heavy-tailed random vectors. The object of interest is the expected top-$k$ Euclidean norm of the sample average, which includes the expected coordinate-wise maximum as the special case $k=1$. Under coordinatewise variance constraints and tail-envelope constraints, the worst-case value is characterized up to universal constants over the class of distributions satisfying a finite $q$:th envelope moment condition. Analogous bounds are obtained for the sub-Weibull envelope class and the marginal sub-Weibull class.
\end{abstract}

\section{Introduction}
Maximal inequalities for sums of independent random vectors are a basic tool in high-dimensional probability and statistics.  Classical bounds for suprema of empirical processes and Banach-space-valued sums are often expressed through entropy or type constants; see, for example, \citet{Ledoux2011Probability}, \citet{Dudley1999Uniform}, and \citet{Gine2016Mathematical}. For independent summands, Bennett's inequality \citep{Bennett1962} gives the standard bounded-envelope benchmark. In the literature on high-dimensional central limit theorems, high-dimensional bootstrap theory, and high-dimensional mean estimation, the relevant object is typically a maximum over coordinates, and the dependence on the variance, dimension, and tail size has to be kept explicit; see \citet{Chernozhukov2017Central} and the more recent refinements of \citet{Chernozhukov2023Nearly}.

The object of interest is
\begin{equation*}
    \Eb\left\|\frac1n\sum_{i=1}^nX_i\right\|_{(k),2},
\end{equation*}
where, for $x=(x(1),\ldots,x(p))^\top\in\Real^p$ and $k\in[p]$,
\begin{equation*}
    \norm{x}_{(k),2}
    =
    \sup_{J\subseteq[p],\ |J|\le k}
    \left(\sum_{j\in J}x(j)^2\right)^{1/2}.
\end{equation*} Equivalently, if $|x|_{(1)}\geq\ldots\geq|x|_{(p)}$ denotes the decreasing rearrangement of the absolute coordinates, then
\begin{equation*}
    \norm{x}_{(k),2} = \left(\sum_{j=1}^k|x|_{(j)}^2\right)^{1/2}.
\end{equation*}
This norm is the dual of the $k$-support norm of \citet{Argyriou2012sparse}. The case $k=1$ corresponds to the coordinate-wise maximum, the object most commonly used in high-dimensional maximal inequalities. Allowing $k>1$ replaces the largest coordinate by the Euclidean norm of the largest $k$ coordinates. This gives a quantitative interpolation between a maximum-type functional and a sparse Euclidean functional. Closely related ordered-coordinate statistics (or $L$-statistics) appear in adaptive high-dimensional testing, where one orders marginal statistics and aggregates a prescribed number of leading terms; see, for example, \citet{ma2024adaptivelstatisticshighdimensional}.

The formulation is worst-case over distribution classes.  This point of view separates the effect of the coordinate-wise variance from the effect of the tail envelope and gives finite-sample benchmarks for maximal inequalities. 

\section{Finite-moment envelope classes}\label{sec:finite-moment}
For $p\geq 1$, let $\Pc(\Real^p)$ denote the class of probability distributions on $\Real^p$.  For $n\geq1$, $q\geq2$, $\sigma>0$, and $B>0$, define $\Pc_{n,p}(q,\sigma,B)$ as the class of product measures $P^n=P_1\otimes\cdots\otimes P_n\in\Pc(\Real^p)^{\otimes n}$ such that, for $X_i\sim P_i$,
\begin{equation*}
    \Eb X_i=0_p,
    \qquad
    V(P^n):=\max_{1\leq j\leq p}\frac1n\sum_{i=1}^n \Eb X_i(j)^2\leq \sigma^2,
\end{equation*}
and
\begin{equation*}
    \Mc_q(P^n):=\frac1n\sum_{i=1}^n\Eb\max_{1\leq j\leq p}|X_i(j)|^q\leq B^q.
\end{equation*}
The quantity of interest is
\begin{equation}\label{eq:quantity_of_interest}
    \Ec_{n,p,k,q}(\sigma,B):=
    \sup_{P^n\in\Pc_{n,p}(q,\sigma,B)}
    \Eb_{P^n}\left\|\frac1n\sum_{i=1}^n X_i\right\|_{(k),2}.
\end{equation}
The envelope moment automatically controls the variance. Indeed, by Jensen's inequality, for $q\ge2$,
\begin{equation*}
    \max_{1\leq j\leq p} \left(\frac{1}{n}\sum_{i=1}^n\Eb[X_i(j)^2]\right)^{1/2}\leq
    \max_{1\leq j\leq p} \left(\frac{1}{n}\sum_{i=1}^n\{\Eb[X_i(j)^2]\}^{q/2}\right)^{1/q}\leq
    \left(\frac{1}{n}\sum_{i=1}^n \Eb\max_{1\leq j\leq p}|X_i(j)|^q\right)^{1/q}.
\end{equation*} Thus $\Mc_q(P^n)\leq B^q$ implies $V(P^n)\leq B^2$, and hence
\begin{equation*}
    \Pc_{n,p}(q,\sigma,B)=\Pc_{n,p}(q,\sigma\wedge B,B),
    \qquad
    \Ec_{n,p,k,q}(\sigma,B)=\Ec_{n,p,k,q}(\sigma\wedge B,B).
\end{equation*}

The iid subclass is obtained by restricting $P^n$ to product measures of the form $P^{\otimes n}$ with $P\in\Pc_{1,p}(q,\sigma,B)$. Formally, we define
\begin{equation*}
    \Ec^*_{n,p,k,q}(\sigma,B)
    :=
    \sup_{P\in\Pc_{1,p}(q,\sigma,B)}
    \Eb_{P^{\otimes n}}\left\|
    \frac1n\sum_{i=1}^n X_i
    \right\|_{(k),2},
\end{equation*}
where $X_1,\ldots,X_n$ are iid with law $P$, $\Eb X=0_p$, $\max_j\Eb X(j)^2\leq\sigma^2$, and $\Eb\max_j|X(j)|^q\leq B^q$.

\begin{proposition}\label{prop:iid_reduction_sparse}
Let $n\geq1$, $p\geq1$, $k\in[p]$, $q\geq2$, and $\sigma,B>0$. Then
\begin{equation*}
    \Ec_{n,p,k,q}^*(\sigma,B)\leq\Ec_{n,p,k,q}(\sigma,B)\leq16\Ec_{n,p,k,q}^*(\sigma,B).
\end{equation*}
\end{proposition}

The proof of Proposition~\ref{prop:iid_reduction_sparse} follows the argument in the proof of Theorem~2.1 of \cite{basu2025maximalinequalitiesindependentrandom}.

\medskip
For the characterization of $\Ec_{n,p,k,q}$ and $\Ec^*_{n,p,k,q}$, set
\begin{equation*}
    \Lambda_q(k)
    =
    {\frac{B^2}{(\sigma\wedge B)^2}}
    \left({\frac{\log(2p/k)}{n}}\right)^{1-2/q}.
\end{equation*} 

We first state the bounds for $k=1$.

\begin{theorem}\label{thm:maximal_upper}
For all $n\ge1$, $p\ge1$, $q\ge2$, and $\sigma,B>0$,
\begin{align}\label{eq:maximal_upper_small}
    \Ec_{n,p,1,q}(\sigma,B)
    \le
    \min\bigg\{B,~&
    9(\sigma\wedge B)\sqrt{\frac{\log(2p)}{n}}\mathbf{1}(\log\Lambda_q(1)\le1) \nonumber\\
    &+ 8B\left({\frac{\log(2p)}{n\log\Lambda_q(1)}}\right)^{1-1/q} \mathbf{1}(\log\Lambda_q(1)>1)\bigg\}.
\end{align}
\end{theorem}

\begin{remark}[Comparison with Theorem 4.3 of \citet{basu2025maximalinequalitiesindependentrandom}]Theorem~4.3 of \citet{basu2025maximalinequalitiesindependentrandom} implies, up to universal constants,
\begin{align}
    \Ec_{n,p,1,q}(\sigma,B)
    \lesssim
    \min\Bigg\{
    B,\,
    &(\sigma\wedge B)\sqrt{\frac{\log(2p)}{n}}
    \mathbf 1\left(\log \Lambda_q(1)\leq\frac{q-2}{q}\right)\nonumber\\
    &+B\left(\frac{\log(2p)}{n\log \Lambda_q(1)}\right)^{1-1/q}
    \mathbf 1\left(\log \Lambda_q(1)>\frac{q-2}{q}\right)
    \Bigg\}.\label{eq:BK}
\end{align}
Thus the two bounds in \eqref{eq:maximal_upper_small} and \eqref{eq:BK} agree when $\log \Lambda_q(1)\le (q-2)/q$ or $\log \Lambda_q(1)>1$.  In the intermediate range $(q-2)/q<\log \Lambda_q(1)\le1$, one has
\begin{equation*}
    \frac{
    B\left(\frac{\log(2p)}{n\log \Lambda_q(1)}\right)^{1-1/q}
    }{
    (\sigma\wedge B)\sqrt{\frac{\log(2p)}{n}}
    }
    =
    \frac{\exp(\log \Lambda_q(1)/2)}{\log \Lambda_q(1)^{1-1/q}}
    \geq 1.
\end{equation*}
Hence, Theorem~\ref{thm:maximal_upper} keeps the smaller variance term in this range. This improvement over Theorem~4.3 of \citet{basu2025maximalinequalitiesindependentrandom} is more pronounced as $q\downarrow2$.
\end{remark}

\begin{theorem}\label{thm:maximal_lower}
For all $n\ge1$, $p\ge1$, $q\ge2$, and $\sigma,B>0$, if $\log\Lambda_q(1)\le1$, then
\begin{equation}\label{eq:maximal_lower_small}
    \Ec^*_{n,p,1,q}(\sigma,B)
    \ge
    \frac{1}{\sqrt{\log 4}}
    \min\left\{B,
    (\sigma\wedge B)\sqrt{\frac{\log(2p)}{n}}\right\}.
\end{equation}
If $\log\Lambda_q(1)>1$ and $\log(2p)\ge \log\Lambda_q(1)$, then
\begin{equation}\label{eq:maximal_lower_large}
    \Ec^*_{n,p,1,q}(\sigma,B)
    \ge
    \frac{1}{640}
    \min\left\{B,
    B\left({\frac{\log(2p)}{n\log\Lambda_q(1)}}\right)^{1-1/q}\right\}.
\end{equation}
\end{theorem}
If $B\geq (\sigma\wedge B)\sqrt{\log(2p)/n}$, then the lower bound in \eqref{eq:maximal_lower_small} is attained via the random vector of independent Rademacher coordinates. In such a case, the numerical constant $1/\sqrt{\log 4}$ is sharp in a sense that
\begin{equation*}
    \inf_{n,p\geq 1}\frac{\Eb\max_{1\leq j\leq p}\Abs{n^{-1}\sum_{i=1}^n\epsilon_{ij}}}{\min\Set{1,\sqrt{\log(2p)/n}}} = 1/\sqrt{\log 4},
\end{equation*} for independent Rademacher $\epsilon_{ij}$'s \citep{chang2026notesconstantsmaximarademacher}. If instead $B< (\sigma\wedge B)\sqrt{\log(2p)/n}$, the random vector of independent coordinates, each of which is iid and symmetrically supported on $\{-1,0,1\}$ with properly assigned weights, attain the lower bound in \eqref{eq:maximal_lower_small}. Under the stated assumption for \eqref{eq:maximal_lower_large}, when $B\geq B(\log(2p)/(n\log \Lambda_q(1)))^{1-1/q}$, \eqref{eq:maximal_lower_large} is attained by the asymmetric three-point distributions supported on $\{-\rho,0,1-\rho\}$ for some $\rho\in[0,1/2]$, otherwise, again by symmetric three-point distribution supported on $\{-1,0,1\}$. See Appendix~\ref{app:proofs-moment} for a detailed construction.

The condition $\log(2p)\ge\log\Lambda_q(1)$ for $\eqref{eq:maximal_lower_large}$ cannot be replaced by $\log(2p)\ge c\log\Lambda_q(1)$ for a fixed $c<1$ without changing the conclusion.

\begin{proposition}\label{prop:sharpness}
Fix $q\ge2$ and $c\in(0,1)$. There is no positive constant $C=C(c,q)$ such that, for all $n,p\ge1$ and $\sigma,B>0$, the conditions $\log\Lambda_q(1)>1$ and $\log(2p)\ge c\log\Lambda_q(1)$ imply
\begin{equation*}
    \Ec^*_{n,p,1,q}(\sigma,B)
    \ge
    C\min\left\{B,
    B\left({\frac{\log(2p)}{n\log\Lambda_q(1)}}\right)^{1-1/q}\right\}.
\end{equation*}
\end{proposition}

We next pass from $k=1$ to general $k\in[p]$. The upper bound follows from Lemma~\ref{lem:order}, which gives
\begin{equation*}
    \Ec_{n,p,k,q}(\sigma,B)\lesssim \sqrt{k}\,\Ec_{n,\lceil p /k \rceil,1,q}(\sigma,B),
\end{equation*}
using the random partitioning argument. Here and below, $\lesssim$ denotes inequality up to a universal constant.

\begin{corollary}\label{cor:upper}
For all $n\ge1$, $p\ge1$, $k\in[p]$, $q\ge2$, and $\sigma,B>0$, there is a universal constant $C>0$ such that 
\begin{align*}
    \Ec_{n,p,k,q}(\sigma,B)
    \le
    \sqrt{k}\min\bigg\{B,~&
    C(\sigma\wedge B)\sqrt{\frac{\log(2p/k)}{n}}\mathbf{1}(\log\Lambda_q(k)\le1) \nonumber\\
    &+ CB\left({\frac{\log(2p/k)}{n\log\Lambda_q(k)}}\right)^{1-1/q} \mathbf{1}(\log\Lambda_q(k)>1)\bigg\}.
\end{align*}
\end{corollary}

The lower bound follows from a simple argument. If the law of $\lfloor p/k\rfloor$-dimensional random vector $X^\circ$ belongs to $\Pc_{n,\lfloor p/k\rfloor}(q,\sigma,B)$ then the law of the following $p$-dimensional random vector belongs to $\Pc_{n,p}(q,\sigma,B)$: 
\begin{equation*}
X((r-1)\lceil p /k \rceil+\ell)=X^\circ(\ell),
\qquad
1\le r\le \floor{p/\lceil p /k \rceil},\quad 1\le \ell\le \lceil p /k \rceil,
\end{equation*}
and $X(j)=0$ for the remaining coordinates. This construction is used in Lemma~\ref{lem:block} to show that
\begin{equation*}
    \Ec_{n,p,k,q}(\sigma,B)\gtrsim \sqrt{k}\,\Ec_{n,\lfloor p /k \rfloor,1,q}(\sigma,B).
\end{equation*}

\begin{corollary}\label{cor:lower}
For all $n\ge1$, $p\ge1$, $k\in[p]$, $q\ge2$, and $\sigma,B>0$, there is a universal constant $c>0$ such that the following bounds hold.  If $\log\Lambda_q(k)\le1$, then
\begin{equation}\label{eq:main_lower_small}
    \Ec^*_{n,p,k,q}(\sigma,B)
    \ge
    c\sqrt{k}
    \min\left\{B,
    (\sigma\wedge B)\sqrt{\frac{\log(2p/k)}{n}}\right\}.
\end{equation}
If $\log\Lambda_q(k)>1$ and $\log(2p/k)\ge \log\Lambda_q(k)$, then
\begin{equation}\label{eq:main_lower_large}
    \Ec^*_{n,p,k,q}(\sigma,B)
    \ge
    c\sqrt{k}
    \min\left\{B,
    B\left({\frac{\log(2p/k)}{n\log\Lambda_q(k)}}\right)^{1-1/q}\right\}.
\end{equation}
\end{corollary}
Combining Corollaries~\ref{cor:upper} and \ref{cor:lower} yields the following characterization.

\begin{corollary}\label{cor:moment-characterization}
For all $n\ge1$, $p\ge1$, $k\in[p]$, $q\ge2$, and $\sigma,B>0$ such that
\begin{enumerate}[label=(\roman*)]
\item the condition $\log\Lambda_q(k)\le1$ holds, one has
\begin{equation*}
    \Ec_{n,p,k,q}(\sigma,B)\asymp\Ec^*_{n,p,k,q}(\sigma,B)\asymp\sqrt{k}\min\left\{B,
    (\sigma\wedge B)\sqrt{\frac{\log(2p/k)}{n}}\right\};
\end{equation*}
\item the conditions $\log\Lambda_q(k)>1$  and $\log(2p/k)\ge\log\Lambda_q(k)$ hold, one has
\begin{equation*}
    \Ec_{n,p,k,q}(\sigma,B)\asymp\Ec^*_{n,p,k,q}(\sigma,B)\asymp\sqrt{k}\min\left\{B,
    B\left({\frac{\log(2p/k)}{n\log\Lambda_q(k)}}\right)^{1-1/q}\right\}.
\end{equation*}
\end{enumerate}
\end{corollary}

\section{\texorpdfstring{$\ell_\infty$}{l-infinity}-Envelope sub-Weibull classes}\label{sec:envelope-subweibull}
For a non-decreasing function $\Psi:[0,\infty)\to[0,\infty)$ with $\Psi(0)=0$ and a scalar random variable $W$, the Orlicz norm is defined as 
\begin{equation*}
    \norm{W}_{\Psi}
    :=
    \inf\left\{c>0:
    \Eb\left[\Psi\left(\frac{|W|}{c}\right)\right]\leq 1
    \right\},
\end{equation*}
where the infimum over an empty set is interpreted as $\infty$. Background on Orlicz norms and generalized Bernstein--Orlicz norms can be found in \citet{Krasnosel1961Convex}, \citet{Dudley1999Uniform}, \citet{kuchibhotla2022moving}, and \citet{bong2023tight}. In particular, by taking $\psi_\alpha(x)=\exp(x^\alpha)-1$, $\alpha>0$, the sub-Weibull norm is defined as
\begin{equation*}
    \norm{W}_{\psi_\alpha}
    :=
    \inf\left\{c>0:
    \Eb\exp\left[\left(\frac{|W|}{c}\right)^\alpha\right]\leq 2
    \right\}.
\end{equation*}

For $n\geq1$, $p\geq1$, $\alpha>0$, and $\sigma,K>0$, let $\Pc_{n,p}^{\psi_\alpha,\infty}(\sigma,K)$ be the class of product measures $P^n=P_1\otimes\cdots\otimes P_n$ such that, for $X_i\sim P_i$, $\Eb X_i=0_p$ and
\begin{equation}\label{eq:Weibull_max_class}
    \max_{1\leq j\leq p}
    \frac1n\sum_{i=1}^n\Eb X_i(j)^2
    \leq \sigma^2,
    \qquad
    \max_{1\leq i\leq n}
    \norm{\max_{1\leq j\leq p}|X_i(j)|}_{\psi_\alpha}
    \leq K.
\end{equation}
For $k\in[p]$, the quantities of interest are defined as
\begin{align*}
    \Ec_{n,p,k,\psi_\alpha}^{\infty}(\sigma,K)
    &:={}
    \sup_{P^n\in\Pc_{n,p}^{\psi_\alpha,\infty}(\sigma,K)}
    \Eb_{P^n}
    \left\|
    \frac1n\sum_{i=1}^nX_i
    \right\|_{(k),2},\\
    \Ec_{n,p,k,\psi_\alpha}^{\infty,*}(\sigma,K)
    &:={}
    \sup_{P\in\Pc_{1,p}^{\psi_\alpha,\infty}(\sigma,K)}
    \Eb_{P^{\otimes n}}
    \left\|
    \frac1n\sum_{i=1}^nX_i
    \right\|_{(k),2}.
\end{align*}
The argument in Proposition~\ref{prop:iid_reduction_sparse} similarly applies and yields that
\begin{equation*}
    \Ec_{n,p,k,\psi_\alpha}^{\infty}(\sigma,K)\asymp\Ec_{n,p,k,\psi_\alpha}^{\infty,*}(\sigma,K).
\end{equation*}

For $k\in[p]$, define
\begin{align*}
    \Dc_{\alpha,k}
    :=
    \Bigg\{
    \Delta:&~
    1\vee\frac{\log(2p/k)}{n}
    \leq \Delta\leq \log(2p/k),\\
    &~K^2\frac{\log(2p/k)}{n\Delta}
    \left[
    \log\left(e+\frac{n\Delta}{\log(2p/k)}\right)
    \right]^{2/\alpha}
    e^{-\Delta}
    \leq \sigma^2
    \Bigg\}.
\end{align*}

The following theorem gives the corresponding upper and lower bounds.

\begin{theorem}\label{thm:sub_orlicz_maximal}
For every $\alpha>0$, there are constants $0<c_\alpha\le C_\alpha<\infty$, depending only on $\alpha$, such that for every $n\geq1$, $p\geq1$, and $\sigma,K>0$,
\begin{align*}
    \Ec_{n,p,1,\psi_\alpha}^{\infty}(\sigma,K)
    &\leq
    C_\alpha
    \min\Bigg\{
    K,
    \max\bigg\{(\sigma\wedge K)\sqrt{\frac{\log(2p)}{n}}, \\
    &\hspace{3.3cm}
    K\frac{\log(2p)}{n}
    \left[
    \log\left(e+\frac{n}{\log(2p)}\right)
    \right]^{1/\alpha}\bigg\}\Bigg\}.
\end{align*}
Moreover,
\begin{align*}
    \Ec_{n,p,1,\psi_\alpha}^{\infty,*}(\sigma,K)
    &\geq
    c_\alpha\min\Bigg\{
    K,
    \max\bigg\{(\sigma\wedge K)\sqrt{\frac{\log(2p)}{n}}, \\
    &\hspace{2.5cm}
    \sup_{\Delta\in\Dc_{\alpha,1}}
    K\frac{\log(2p)}{n\Delta}
    \left[
    \log\left(e+\frac{n\Delta}{\log(2p)}\right)
    \right]^{1/\alpha}\bigg\}\Bigg\}.
\end{align*}The supremum over an empty set is interpreted as $-\infty$.
\end{theorem}

\begin{remark}
Theorem~\ref{thm:sub_orlicz_maximal} is the sub-Weibull envelope analogue of Theorems~\ref{thm:maximal_upper} and \ref{thm:maximal_lower}. Indeed, the condition
$\norm{\max_{1\leq j\leq p}|X_i(j)|}_{\psi_\alpha}\leq K$ implies that $$\Pb\left(\max_{1\leq j\leq p}|X_i(j)|>Kt\right)\leq 2e^{-t^\alpha},$$ for $t\geq 0$, and integration gives 
\begin{equation}\label{eq:psi_to_moment_envelope}
    \left\{
    \Eb\max_{1\leq j\leq p}|X_i(j)|^q
    \right\}^{1/q}
    \leq
    K\left\{2\Gamma\left(1+\frac q\alpha\right)\right\}^{1/q}.
\end{equation} Applying
Theorem~\ref{thm:maximal_upper} with the moment envelope in \eqref{eq:psi_to_moment_envelope} gives a valid upper bound. In particular, the elementary optimization (see Section~\ref{app:proof-coordinate-wise} for a detailed derivation) that for $x\geq 0$,
\begin{equation}\label{eq:optimization}
    \inf_{q\geq2}\left\{2\Gamma\left(1+\frac{q}{\alpha}\right)\right\}^{1/q}x^{1-1/q}\asymp_\alpha \min\Set{\sqrt{x},x\log^{1/\alpha}\left(e/x\right)},
\end{equation} and taking $x = \log(2p)/(n\log\Lambda_q(1))$ results in the scaling of the upper bound in Theorem~\ref{thm:sub_orlicz_maximal}. The lower bound uses the same type of construction as in Theorem~\ref{thm:maximal_lower}.
\end{remark}

The extension to general $k$ follows from Lemmas~\ref{lem:order} and \ref{lem:block}.

\begin{corollary}\label{cor:sub_orlicz}
For every $\alpha>0$, there are constants $0<c_\alpha\le C_\alpha<\infty$, depending only on $\alpha$, such that for every $n\geq1$, $p\geq1$, $k\in[p]$, and $\sigma,K>0$,
\begin{align}
    \Ec_{n,p,k,\psi_\alpha}^{\infty}(\sigma,K)
    &\leq
    C_\alpha\sqrt{k}
    \min\Bigg\{
    K,
    \max\bigg\{(\sigma\wedge K)\sqrt{\frac{\log(2p/k)}{n}},\nonumber\\
    &\qquad\qquad K\frac{\log(2p/k)}{n}
    \left[
    \log\left(e+\frac{n}{\log(2p/k)}\right)
    \right]^{1/\alpha}\bigg\}\Bigg\}.
    \label{eq:sub_orlicz_upper_readable}
\end{align}
Moreover,
\begin{align}
    \Ec_{n,p,k,\psi_\alpha}^{\infty,*}(\sigma,K)
    &\geq
    c_\alpha\sqrt{k}\min\Bigg\{
    K,
    \max\bigg\{(\sigma\wedge K)\sqrt{\frac{\log(2p/k)}{n}},\nonumber\\
    &\qquad\sup_{\Delta\in\Dc_{\alpha,k}}
    K\frac{\log(2p/k)}{n\Delta}
    \left[
    \log\left(e+\frac{n\Delta}{\log(2p/k)}\right)
    \right]^{1/\alpha}\bigg\}\Bigg\},
    \label{eq:sub_orlicz_lower_readable}
\end{align}

\end{corollary}

When $1\in\Dc_{\alpha,k}$, the bounds in Corollary~\ref{cor:sub_orlicz} match up to constants depending only on $\alpha$. To see this, we note that for every $a\geq1$ and $\Delta\geq1$,
\begin{equation}\label{eq:delta-log-bound}
    \frac{1}{\Delta}\left[\log(e+a\Delta)\right]^{1/\alpha}
    \leq
    \max\left\{1,e^{1-1/\alpha}\alpha^{-1/\alpha}\right\}
    \left[\log(e+a)\right]^{1/\alpha}.
\end{equation}
Applying this with $a=n/\log(2p/k)$ gives the next characterization.

\begin{corollary}\label{cor:sub-orlicz-characterization}
Suppose that
\begin{equation*}
    K^2\frac{\log(2p/k)}{n}
    \left[
    \log\left(e+\frac{n}{\log(2p/k)}\right)
    \right]^{2/\alpha}
    \leq e\sigma^2.
\end{equation*}
Then
\begin{align*}
    &\qquad\Ec_{n,p,k,\psi_\alpha}^{\infty}(\sigma,K)
    \asymp
    \Ec_{n,p,k,\psi_\alpha}^{\infty,*}(\sigma,K)\\
    &\asymp_\alpha
    \sqrt{k}
    \min\Bigg\{
    K,
    \max\bigg\{(\sigma\wedge K)\sqrt{\frac{\log(2p/k)}{n}},
    K\frac{\log(2p/k)}{n}
    \left[
    \log\left(e+\frac{n}{\log(2p/k)}\right)
    \right]^{1/\alpha}\bigg\}\Bigg\}.
\end{align*}
\end{corollary}

\section{Marginal sub-Weibull classes}\label{sec:coordinate-wise}
Let $\Pc_{n,p}^{\psi_\alpha,{\rm m}}(\sigma,K)$ be the class of product measures $P^n=P_1\otimes\cdots\otimes P_n$ such that, for $X_i\sim P_i$, $\Eb X_i=0_p$ and
\begin{equation}\label{eq:Weibull_marginal_class}
    \max_{1\leq j\leq p}
    \frac1n\sum_{i=1}^n\Eb X_i(j)^2
    \leq \sigma^2,
    \qquad
    \max_{1\leq i\leq n}\max_{1\leq j\leq p}
    \norm{X_i(j)}_{\psi_\alpha}
    \leq K.
\end{equation}
Define $\Ec_{n,p,k,\psi_\alpha}^{\rm m}(\sigma,K)$ and $\Ec_{n,p,k,\psi_\alpha}^{{\rm m},*}(\sigma,K)$ by replacing the envelope class in Section~\ref{sec:envelope-subweibull} by $\Pc_{n,p}^{\psi_\alpha,{\rm m}}(\sigma,K)$. Again, following the Proof of Proposition~\ref{prop:iid_reduction_sparse}, one may deduce that
\begin{equation*}
    \Ec_{n,p,k,\psi_\alpha}^{\rm m}(\sigma,K)
    \asymp
    \Ec_{n,p,k,\psi_\alpha}^{{\rm m},*}(\sigma,K).
\end{equation*}

The sub-Weibull envelope condition implies the coordinate-wise sub-Weibull condition, i.e., $\Pc_{n,p}^{\psi_\alpha,\infty}(\sigma,K)\subset\Pc_{n,p}^{\psi_\alpha,{\rm m}}(\sigma,K)$. Also, a union bound gives the elementary inclusion that
\begin{equation*}
    \Pc_{n,p}^{\psi_\alpha,{\rm m}}(\sigma,K)
    \subset
    \Pc_{n,p}^{\psi_\alpha,\infty}
    \left(\sigma,K\{\log_2(2p)\}^{1/\alpha}\right).
\end{equation*} Indeed, let $A=\log_2(2p)$ and $\max_j\norm{X(j)}_{\psi_\alpha}\le K$, then
\begin{align*}
    \Eb\exp\left[\left(\frac{\max_{1\leq j\leq p}|X(j)|}{KA^{1/\alpha}}\right)^\alpha\right]
    &=\Eb\max_{1\leq j\leq p}\exp\left[\left(\frac{|X(j)|}{KA^{1/\alpha}}\right)^\alpha\right] \\
    &\leq \Eb\left\{\sum_{j=1}^p\exp\left[\left(\frac{|X(j)|}{K}\right)^\alpha\right]\right\}^{1/A}
    \leq 2.
\end{align*} Therefore, we always have
\begin{equation*}
    \Ec_{n,p,k,\psi_\alpha}^{\infty}(\sigma,K)\leq \Ec_{n,p,k,\psi_\alpha}^{\rm m}(\sigma,K)\leq \Ec_{n,p,k,\psi_\alpha}^{\infty}(\sigma,K\{\log_2(2p)\}^{1/\alpha}).
\end{equation*} However, the above crude inequalities do not give a sharp characterization of $\Ec_{n,p,k,\psi_\alpha}^{\rm m}(\sigma,K)$ since, in particular, the upper bound replaces the marginal scale $K$ by an envelope scale of order $K\log_2^{1/\alpha}(2p)$.

The next theorem gives a sharper upper bound for $\Ec_{n,p,k,\psi_\alpha}^{\rm m}(\sigma,K)$.

\begin{theorem}\label{thm:k1}
There exist a universal constant $c>0$ and a constant $C_\alpha>0$, depending only on $\alpha$, such that for all $n,p\ge1$ and $\sigma,K>0$,
\begin{align}\label{eq:k1_alpha_lt_1}
    \Ec_{n,p,1,\psi_\alpha}^{{\rm m}}(\sigma,K)
    &\leq
    (c\sigma\wedge C_\alpha K)\sqrt{\frac{\log(2p)}{n}}
    +C_\alpha K\frac{1\vee\log(p)}{n}
    \left[\log\left(e+\frac{n}{1\vee\log(p)}\right)\right]^{1/\alpha}\nonumber\\
    &\quad +C_\alpha K\frac{1\vee\log^{1/\alpha}(p)}{n}\mathbf{1}_{(0,1)}(\alpha)
\end{align}
\end{theorem} The upper bound in \eqref{eq:k1_alpha_lt_1} vanishes as long as $\log(p) = o(n^{1\wedge \alpha})$.

\begin{remark}[Comparison with existing maximal inequalities]
The GBO bound in Theorem~3.4 of \cite{kuchibhotla2022moving} (see also Remark~3.2 therein) implies, under the same marginal sub-Weibull condition in \eqref{eq:Weibull_marginal_class}, that
\begin{equation}\label{eq:KC}
    \Eb\left\|\frac{1}{n}\sum_{i=1}^nX_i\right\|_\infty\leq
    C_1\sigma\sqrt{ \frac{\log(ep)}{n} }
    +
    C_2(\alpha,K)\frac{\log^{1/\alpha}(2n)\log^{1/(\alpha\wedge1)}(ep)}{n}.
\end{equation} By contrast, Theorem~\ref{thm:k1} implies that
\begin{equation}\label{eq:ours}
    \Eb\left\|\frac{1}{n}\sum_{i=1}^nX_i\right\|_\infty
    \leq c\sigma\sqrt{\frac{\log(ep)}{n}}
    +C_\alpha K\Bigg[\frac{\log(ep)\log^{1/\alpha}(2n/\log(ep))}{n}+\frac{\log^{1/\alpha}(ep)}{n}\mathbf{1}_{(0,1)}(\alpha)\Bigg],
\end{equation}
by enlarging the constants.  While both bounds in \eqref{eq:KC} and \eqref{eq:ours} track the same variance term, the second term on the right-hand side of \eqref{eq:ours} improves the dependence on $n$ and $p$ in \eqref{eq:KC}. Indeed, 
\begin{align*}
    \frac{\log(ep)\log^{1/\alpha}(2n/\log(ep))}{n}&<\frac{\log^{1/\alpha}(2n)\log^{1/(\alpha\wedge1)}(ep)}{n},\qquad \forall \alpha>0,\\
    \frac{\log^{1/\alpha}(ep)}{n}&<\frac{\log^{1/\alpha}(2n)\log^{1/\alpha}(ep)}{n},\qquad \forall \alpha\in(0,1).
\end{align*}
\end{remark}
\begin{theorem}\label{thm:mar_weibull_lower_maximal}
For every $n\geq1$, $p\geq1$, $\alpha>0$, and $\sigma,K>0$,
\begin{equation}\label{eq:lower_rademacher_maximal}
    \Ec_{n,p,1,\psi_\alpha}^{{\rm m}, *}(\sigma,K)
    \gtrsim (\log2)^{1/\alpha}
    (\sigma\wedge K)
    \min\left\{1,\sqrt{\frac{\log(2p)}{n}}\right\}.
\end{equation}
Let
\begin{equation*}
    a_\alpha
    =
    \left[
    \min\left\{\frac12,\,e\alpha\log\left(2-\frac{1}{2\sqrt2}\right)\right\}
    \right]^{1/\alpha}.
\end{equation*}
If
\begin{equation}\label{eq:centered_bernoulli_feasible_maximal}
    \log(2p)\leq n \qquad\mbox{and}\qquad a_\alpha^2K^2\frac{\log(2p)}{8n}
    \left\{\log\left(\frac{8n}{\log(2p)}\right)\right\}^{2/\alpha}
    \leq
    \sigma^2,
\end{equation}
then
\begin{equation}\label{eq:centered_bernoulli_lower_maximal}
    \Ec_{n,p,1,\psi_\alpha}^{{\rm m}, *}(\sigma,K)
    \gtrsim a_\alpha
    K\frac{\log(2p)}{n}
    \left\{\log\left(\frac{8n}{\log(2p)}\right)\right\}^{1/\alpha}.
\end{equation}
Finally, if
\begin{equation}\label{eq:bernoulli_rademacher_feasible_maximal}
    \frac{K^2\{\log(1+8np)\}^{2/\alpha}}{8np}
    \leq
    \sigma^2,
\end{equation}
then
\begin{equation}\label{eq:bernoulli_rademacher_lower_maximal}
    \Ec_{n,p,1,\psi_\alpha}^{{\rm m}, *}(\sigma,K)
    \gtrsim
    \frac{K\{\log(1+8np)\}^{1/\alpha}}{n}.
\end{equation}
\end{theorem}

\begin{remark} The first two lower bounds in \eqref{eq:lower_rademacher_maximal} and \eqref{eq:centered_bernoulli_lower_maximal}, respectively, match the components in the upper bounds \eqref{eq:k1_alpha_lt_1}. It remains to compare 
\eqref{eq:bernoulli_rademacher_lower_maximal}. Since $n,p\ge1$,
\[
    \log(1+8np)
    \leq \log 9+\log n+\log p
    \leq 5(1\vee\log p+\log(e+n/(1\vee\log p))).
\]
If $\alpha\geq1$, then
\[
    \{\log(1+8np)\}^{1/\alpha}
    \leq
    2\cdot5^{1/\alpha}(1\vee\log p)\,\log^{1/\alpha}\big(e+n/(1\vee\log p)\big).
\]
If $0<\alpha<1$, then
\[
    \{\log(1+8np)\}^{1/\alpha}
    \leq
    5^{1/\alpha}2^{1/\alpha-1}
    \left\{
        (1\vee\log p)^{1/\alpha}+(1\vee\log p)\,\log^{1/\alpha}\big(e+n/(1\vee\log p)\big)
    \right\}.
\]
After multiplying by $K/n$, the right-hand side of
\eqref{eq:bernoulli_rademacher_lower_maximal} is bounded by the last two terms on the right-hand side of \eqref{eq:k1_alpha_lt_1}, up to a constant depending only on
$\alpha$.
\end{remark}

The three lower-bound terms are obtained from Rademacher, three-point, and centered Bernoulli-type constructions, respectively.  The proof is included in Appendix~\ref{app:proof-coordinate-wise}.

The corresponding bounds for the top-$k$ Euclidean norm follow.

\begin{corollary}\label{cor:mar_weibull_upper}
There exist constants $c>0$ and $C_\alpha>0$, depending only on $\alpha$, such that for all $n\ge1$, $p\ge1$, $k\in[p]$, and $\sigma,K>0$,
\begin{align}\label{eq:cor:mar_weibull_upper}
    \Ec_{n,p,k,\psi_\alpha}^{{\rm m}}(\sigma,K)
    \leq \sqrt{k}\Bigg[&(c\sigma\wedge C_\alpha K)\sqrt{\frac{\log(2p/k)}{n}}
    +C_\alpha K\frac{1\vee\log(p/k)}{n}
    \log^{1/\alpha}\left(e+\frac{n}{1\vee\log(p/k)}\right)\nonumber\\
    &\quad +C_\alpha K\frac{1\vee\{\log(p/k)\}^{1/\alpha}}{n}\mathbf{1}_{(0,1)}(\alpha)\Bigg].
\end{align}
\end{corollary}

\begin{corollary}\label{cor:mar_weibull_lower}
For every $n\geq1$, $p\geq1$, $k\in[p]$, $\alpha>0$, and $\sigma,K>0$,
\begin{equation}\label{eq:lower_rademacher_sparse}
    \Ec_{n,p,k,\psi_\alpha}^{{\rm m}, *}(\sigma,K)
    \gtrsim (\log2)^{1/\alpha}\sqrt{k}
    (\sigma\wedge K)
    \min\left\{1,\sqrt{\frac{\log(2\ceil{p/k})}{n}}\right\}.
\end{equation}
If
\begin{equation}\label{eq:centered_bernoulli_feasible}
    \log(2\ceil{p/k})\leq n \qquad\mbox{and}\qquad a_\alpha^2K^2\frac{\log(2\ceil{p/k})}{8n}
    \left\{\log\left(\frac{8n}{\log(2\ceil{p/k})}\right)\right\}^{2/\alpha}
    \leq
    \sigma^2,
\end{equation}
then
\begin{equation}\label{eq:centered_bernoulli_lower_sparse}
    \Ec_{n,p,k,\psi_\alpha}^{{\rm m}, *}(\sigma,K)
    \gtrsim a_\alpha \sqrt{k}
    K\frac{\log(2\ceil{p/k})}{n}
    \left\{\log\left(\frac{8n}{\log(2\ceil{p/k})}\right)\right\}^{1/\alpha}.
\end{equation}
If
\begin{equation}\label{eq:bernoulli_rademacher_feasible}
    \frac{K^2\{\log(1+8n\ceil{p/k})\}^{2/\alpha}}{8n\ceil{p/k}}
    \leq
    \sigma^2,
\end{equation}
then
\begin{equation}\label{eq:bernoulli_rademacher_lower_sparse}
    \Ec_{n,p,k,\psi_\alpha}^{{\rm m}, *}(\sigma,K)
    \gtrsim\sqrt{k}
    \frac{K\{\log(1+8n\ceil{p/k})\}^{1/\alpha}}{n}.
\end{equation}
\end{corollary}

The following corollary records the resulting characterization of $\Ec_{n,p,k,\psi_\alpha}^{{\rm m}}(\sigma,K)$.
\begin{corollary}\label{cor:mar_weibull_characterization}
Let $n\geq1$, $p\geq1$, $k\in[p]$, $\alpha>0$, and $\sigma,K>0$.
Suppose that \eqref{eq:centered_bernoulli_feasible} holds. If
$\alpha\in(0,1)$, suppose in addition that
\eqref{eq:bernoulli_rademacher_feasible} holds. Then both quantities $\Ec_{n,p,k,\psi_\alpha}^{{\rm m}}(\sigma,K)$ and $\Ec_{n,p,k,\psi_\alpha}^{{\rm m},*}(\sigma,K)$ are of the same order, up to constants depending only on
$\alpha$, as the right-hand side of \eqref{eq:cor:mar_weibull_upper}.
\end{corollary}

\appendix
\section{Proofs}

\subsection{Proofs for Section~\ref{sec:finite-moment}}\label{app:proofs-moment}

\begin{proof}[Proof of Proposition~\ref{prop:iid_reduction_sparse}]
The lower bound is immediate because every iid product measure $P^{\otimes n}$ with $P\in\Pc_{1,p}(q,\sigma,B)$ belongs to $\Pc_{n,p}(q,\sigma,B)$.  For the reverse inequality, the proof of Theorem~2.1 of \citet{basu2025maximalinequalitiesindependentrandom} applies with the norm $\norm{\cdot}_{(k),2}$ in place of $\norm{\cdot}_\infty$.
\end{proof}

\begin{proof}[Proof of Theorem~\ref{thm:maximal_upper}]
Theorem~\ref{thm:maximal_upper} follows from Proposition~\ref{prop:sup_upper} below. In particular, if $\log\Lambda_q(1)\leq 1$, then \eqref{eq:sup_upper_general} implies that
\begin{align*}
    \Ec_{n,p,1,q}(\sigma,B)
    \le
    \min\left\{B,\,
    \frac{5}{2}(\sigma\wedge B)\sqrt{\frac{\log(2p)}{n}}
    +\frac{11}{3}B\left({\frac{\log(2p)}{n}}\right)^{1-1/q}\right\}\\
    \leq \min\left\{B,\,
    \left(\frac{5}{2}+\frac{11\sqrt{e}}{3}\right)(\sigma\wedge B)\sqrt{\frac{\log(2p)}{n}}\right\}\leq \min\left\{B,\,
    9(\sigma\wedge B)\sqrt{\frac{\log(2p)}{n}}\right\}.
\end{align*}
\end{proof}

\begin{proposition}\label{prop:sup_upper}
For all $n,p\ge1$, $q\ge2$, and $\sigma,B>0$,
\begin{equation}\label{eq:sup_upper_general}
    \Ec_{n,p,1,q}(\sigma,B)
    \le
    \min\left\{B,\,
    \frac{5}{2}(\sigma\wedge B)\sqrt{\frac{\log(2p)}{n}}
    +\frac{11}{3}B\left({\frac{\log(2p)}{n}}\right)^{1-1/q}\right\}.
\end{equation}
If
\begin{equation}\label{eq:sup_large_condition}
    \log\left\{ {\frac{B^2}{(\sigma\wedge B)^2}}
    \left({\frac{\log(2p)}{n}}\right)^{1-2/q}\right\}>1,
\end{equation}
then
\begin{equation}\label{eq:sup_upper_log}
    \Ec_{n,p,1,q}(\sigma,B)
    \le
    \min\left\{B,
    8B\left({\frac{\log(2p)}{n\log\left\{{\frac{B^2}{(\sigma\wedge B)^2}}({\frac{\log(2p)}{n}})^{1-2/q}\right\}}}
    \right)^{1-1/q}\right\}.
\end{equation}
\end{proposition}

\begin{proof}[Proof of Proposition~\ref{prop:sup_upper}]
It is enough to prove the claim with $\sigma\wedge B$ in place of $\sigma$. Let $X_i$ be independent and admissible in dimension $m$, and write $M_i=\max_{1\le j\le m}|X_i(j)|$. The trivial bound is
\begin{equation}\label{eq:trivial_B}
    \Eb\left\Vert {\frac{1}{n}}\sum_{i=1}^nX_i\right\Vert_\infty
    \le {\frac{1}{n}}\sum_{i=1}^n\Eb M_i
    \le \left({\frac{1}{n}}\sum_{i=1}^n\Eb M_i^q\right)^{1/q}
    \le B.
\end{equation}
For \eqref{eq:sup_upper_general}, set $\tau=B(n/\log(2p))^{1/q}$. Define
\begin{equation*}
    U_i(j)=X_i(j)\mathbf{1}\{M_i\le\tau\},\quad
    Y_i(j)=U_i(j)-\Eb U_i(j),
\end{equation*}
\begin{equation*}
    V_i(j)=X_i(j)\mathbf{1}\{M_i>\tau\}-\Eb[X_i(j)\mathbf{1}\{M_i>\tau\}].
\end{equation*}
Then $X_i(j)=Y_i(j)+V_i(j)$ (because $\Eb[X_i(j)]=0$), $|Y_i(j)|\le2\tau$, and $\sum_i\Eb Y_i(j)^2\le n(\sigma\wedge B)^2$. Lemma \ref{lem:bennett} with a union bound over $2m$ signed coordinates gives, for $u\ge0$,
\begin{equation*}
    \Pb\left(\left\Vert {\frac{1}{n}}\sum_{i=1}^nY_i\right\Vert_\infty
    >(\sigma\wedge B)\sqrt{{\frac{2\{\log(2p)+u\}}{n}}}
    +{\frac{2\tau\{\log(2p)+u\}}{3n}}\right)\le e^{-u}.
\end{equation*}
Integrating this display gives
\begin{align}\label{eq:bounded_general}
\Eb\left\|\frac1n\sum_{i=1}^nY_i\right\|_\infty
    \leq
    (\sigma\wedge B)\sqrt{\frac{2\log(2p)}{n}}+\frac{2\tau\log(2p)}{3n}
    \nonumber\\
    +
    \int_0^\infty e^{-u}
    \left\{
    \frac{\sigma\wedge B}{\sqrt{2n\{\log(2p)+u\}}}
    +\frac{2\tau}{3n}
    \right\}\,du \nonumber\\
    \leq
    \left(\sqrt2+\frac{1}{\sqrt2\log(2p)}\right)
    (\sigma\wedge B)\sqrt{\frac{\log(2p)}{n}} +
    \frac23\left(1+\frac1{\log(2p)}\right)
    B\left(\frac{\log(2p)}{n}\right)^{1-1/q} \nonumber\\
    \leq
    \frac52(\sigma\wedge B)\sqrt{\frac{\log(2p)}{n}}
    +
    \frac53B\left(\frac{\log(2p)}{n}\right)^{1-1/q},
\end{align}
where the last inequality uses $\log(2p)\geq\log2$. Moreover,
\begin{equation*}
    \Eb\left\|\frac1n\sum_{i=1}^nV_i\right\|_\infty
    \leq
    2B\left(\frac{\log(2p)}{n}\right)^{1-1/q}.
\end{equation*}
Combining these bounds with the triangle inequality proves
\begin{equation*}
    \Ec_{n,p,1,q}(\sigma,B)
    \leq
    \min\left\{
    B,\,
    \frac52(\sigma\wedge B)\sqrt{\frac{\log(2p)}{n}}
    +
    \frac{11}{3}B\left(\frac{\log(2p)}{n}\right)^{1-1/q}
    \right\}.
\end{equation*}
Moreover,
\begin{equation}\label{eq:tail_general}
    \Eb\left\Vert {\frac{1}{n}}\sum_{i=1}^nV_i\right\Vert_\infty
    \le {\frac{2}{n}}\sum_{i=1}^n\Eb[M_i\mathbf{1}\{M_i>\tau\}]
    \le 2B^q\tau^{1-q}
    =2B\left({\frac{\log(2p)}{n}}\right)^{1-1/q}.
\end{equation}
Combining \eqref{eq:trivial_B}, \eqref{eq:bounded_general}, and \eqref{eq:tail_general} proves \eqref{eq:sup_upper_general}.

For \eqref{eq:sup_upper_log}, put
\begin{equation*}
    \Delta=
    \log\left\{ {\frac{B^2}{(\sigma\wedge B)^2}}
    \left({\frac{\log(2p)}{n}}\right)^{1-2/q}\right\}>1
\end{equation*}
and truncate at $\tau=B(n\Delta/\log(2p))^{1/q}$. The tail calculation gives
\begin{equation}\label{eq:tail_log}
    \Eb\left\Vert {\frac{1}{n}}\sum_{i=1}^nV_i\right\Vert_\infty
    \le 2B\left({\frac{\log(2p)}{n\Delta}}\right)^{1-1/q}.
\end{equation}
In order to apply Lemma~\ref{lem:bennett}, we note that $K=2\tau$ and
\begin{equation*}
    (\sigma\wedge B)^2
    =B^2\left({\frac{\log(2p)}{n}}\right)^{1-2/q}e^{-\Delta}.
\end{equation*}
For every $v\ge1$,
\begin{equation*}
    {\frac{K\,3vB\{\log(2p)/(n\Delta)\}^{1-1/q}}{(\sigma\wedge B)^2}}
    =6ve^\Delta\Delta^{2/q-1}\ge1.
\end{equation*}
Since $h(x)\ge x\{\log(1+x)-1\}$ and
\begin{equation*}
    \log(1+6ve^\Delta\Delta^{2/q-1})-1
    \ge \Delta-\log\Delta+\log 6-1
    \ge {\frac{5\Delta}{6}},
\end{equation*}
where the last inequality is due to $\frac{x}{6}-\log x + \log 6- 1\geq 0$ for $x\geq 1$, Bennett's inequality and the union bound yield
\begin{equation*}
    \Pb\left(\left\Vert {\frac{1}{n}}\sum_{i=1}^nY_i\right\Vert_\infty
    >3vB\left({\frac{\log(2p)}{n\Delta}}\right)^{1-1/q}\right)
    \le \exp\{-(5v/4-1)\log(2p)\}.
\end{equation*}
Integrating over $v\ge1$ gives
\begin{equation}\label{eq:bounded_log}
    \Eb\left\Vert {\frac{1}{n}}\sum_{i=1}^nY_i\right\Vert_\infty
    \le 6B\left({\frac{\log(2p)}{n\Delta}}\right)^{1-1/q}.
\end{equation}
Combining \eqref{eq:trivial_B}, \eqref{eq:tail_log}, and \eqref{eq:bounded_log} proves \eqref{eq:sup_upper_log}.
\end{proof}

\begin{proof}[Proof of Theorem~\ref{thm:maximal_lower}] 
We first prove \eqref{eq:maximal_lower_small}. Suppose that $\log(2p)\le n$ and take
$X_i(j)=(\sigma\wedge B)\varepsilon_{ij}$ for independent Rademacher variables $\epsilon_{ij}$ ($1\leq i\leq n$ and $1\leq j\leq p$). It is clear that this construction is admissible and Theorem~2.1 of \cite{chang2026notesconstantsmaximarademacher} implies that for all $n,p\geq 1$,
\begin{equation*}
    \Eb\max_{1\le j\le p}\Abs{\frac{1}{n}\sum_{i=1}^n X_i(j)}=(\sigma\wedge B)\Eb\max_{1\le j\le p}\Abs{\frac{1}{n}\sum_{i=1}^n\epsilon_{ij}}\geq \frac{\sigma\wedge B}{\sqrt{2\log 2}}\sqrt{\frac{\log(2p)}{n}}.
\end{equation*} In fact, $(\sqrt{2\log 2})^{-1}$ is the largest numerical constant possible. Next, suppose that $\log(2p)>n$, let
\begin{equation*}
A=\min\left\{B,(\sigma\wedge B)\sqrt{\frac{\log(2p)}{n}}\right\},
\qquad
\eta={\frac{(\sigma\wedge B)^2}{A^2}}\leq 1.
\end{equation*}
Let $X_i(j)=A R_{ij}\varepsilon_{ij}$, where the $R_{ij}$ are Bernoulli$(\eta)$ and the $\varepsilon_{ij}$ are independent Rademacher variables. Then $\Eb X_i(j)=0$, $\Eb X_i(j)^2=(\sigma\wedge B)^2$, and $\max_j|X_i(j)|\le A\le B$. For a fixed $j\in[p]$, the event that
$R_{1j}=\cdots=R_{nj}=1$ and the signs
$\varepsilon_{1j},\ldots,\varepsilon_{nj}$ are all equal has probability $2(\eta/2)^n$ and implies
$n^{-1}\sum_iX_i(j)=A$. Hence
\begin{equation*}
    \Eb\left\Vert {\frac{1}{n}}\sum_{i=1}^nX_i\right\Vert_\infty \ge A\left[1-\{1-2(\eta/2)^n\}^p\right].
\end{equation*}
If $A=(\sigma\wedge B)\sqrt{\log(2p)/n}$, then $\eta=n/\log(2p)$. If $A=B$, then $\eta\ge {\frac{n}{\log(2p)}}.$ Therefore, in both cases, 
\begin{equation*}
    1-\{1-2(\eta/2)^n\}^p\geq 1-
\left\{
1-2\left(\frac{n}{2\log(2p)}\right)^n
\right\}^p.
\end{equation*} We claim that whenever $\log(2p)>n$,
\begin{equation*}
    1-
\left\{
1-2\left(\frac{n}{2\log(2p)}\right)^n
\right\}^p \geq 1-\left(1-\frac{2}{\{\log 12\}^2}\right)^6=:c_0>0.9.
\end{equation*} For $n\ge3$,
\begin{equation*}
    2p\left(\frac{n}{2\log(2p)}\right)^n
=
\exp\left[
n\left\{
\frac{\log(2p)}{n}
-\log\left(\frac{2\log(2p)}{n}\right)
\right\}
\right]
\ge
\left(\frac e2\right)^n
\ge
\left(\frac e2\right)^3 .
\end{equation*} Hence
\begin{equation*}
    1-
\left\{
1-2\left(\frac{n}{2\log(2p)}\right)^n
\right\}^p
\ge
1-\exp\left\{-\left(\frac e2\right)^3\right\}
>c_0
\end{equation*} For $n=1$, necessarily $p\ge2$, and
\[
1-\left(1-\frac1{\log(2p)}\right)^p\ge c_0 .
\]
Indeed, for $p\ge7$ this follows from the monotonicity of
$p/\log(2p)$ and the bound $1-x\le e^{-x}$; the remaining cases $2\le p\le6$ are checked directly, with the minimum attained at $p=3$, where the value is
\[
1-\left(1-\frac1{\log 6}\right)^3
>
c_0 .
\]
For $n=2$, necessarily $p\ge4$, and
\[
1-\left(1-\frac{2}{\{\log(2p)\}^2}\right)^p\ge c_0 .
\]
For $p\ge12$, this follows from the monotonicity of
$p/\{\log(2p)\}^2$ on $\log(2p)>2$ and the bound $1-x\le e^{-x}$.
For $4\le p\le11$, direct evaluation gives the minimum at $p=6$,
which is exactly $c_0$. Thus
\begin{equation*}
    \Eb\left\Vert \frac1n\sum_{i=1}^n X_i\right\Vert_\infty
\ge c_0 A > 0.9A\geq \frac{1}{\sqrt{2\log 2}}A.
\end{equation*}

Next, we prove \eqref{eq:maximal_lower_large} under the assumption that $\log\Lambda_q(1)>1$ and $\log(2p)\geq \log\Lambda_q(1)$. First suppose that $\log(2p)/n\ge\log\Lambda_q(1)$. Let $Y_{ij}$ be iid
Bernoulli$(\rho)$ variables and set
\begin{equation*}
    X_i(j)=B(Y_{ij}-\rho),
\qquad
\rho=\min\left\{ {\frac{1}{4}},{\frac{(\sigma\wedge B)^2}{B^2}}\right\}.
\end{equation*} Then $\Eb X_i(j)=0$, $\Eb X_i(j)^2\le(\sigma\wedge B)^2$, and
$\max_j|X_i(j)|\le B$. If $\rho<1/4$, then
\begin{equation*}
\log\rho=-\log\Lambda_q(1)+\left(1-{\frac{2}{q}}\right)\log{\frac{\log(2p)}{n}}\ge-\log\Lambda_q(1).
\end{equation*}
Therefore $p\rho^n\ge1/2$. With probability at least $1-e^{-1/2}$, there is
a coordinate for which $Y_{1j}=\cdots=Y_{nj}=1$, and on this event
\begin{equation*}
    {\frac{1}{n}}\sum_{i=1}^nX_i(j)=B(1-\rho)\ge {\frac{3B}{4}}.
\end{equation*}
Hence the desired lower bound follows in this case. If $\rho=1/4$, then the event that the first $\lceil n/2\rceil$ variables in
a coordinate equal one implies
\begin{equation*}
    {\frac{1}{n}}\sum_{i=1}^nX_i(j)\ge {\frac{B}{4}}.
\end{equation*} Since $\log\Lambda_q(1)>1$ and $\log(2p)/n\ge\log\Lambda_q(1)$, we have $\log(2p)>n$, and hence
\begin{equation*}
p4^{-\lceil n/2\rceil}\ge {\frac{e}{8}}.
\end{equation*}
Thus the preceding event occurs for some coordinate with probability at least
$1-\exp(-e/8)$. This again gives the desired lower bound.

It remains to consider the case $\log(2p)/n<\log\Lambda_q(1)$. Let $R_i$ be iid
Bernoulli variables with mean $\log(2p)/(n\log\Lambda_q(1))$, let $Y_{ij}$ be iid
Bernoulli variables independent of the $R_i$'s, and define
\begin{equation*}
    N=\sum_{i=1}^nR_i,
\qquad
Z_j=\sum_{i=1}^nR_iY_{ij}.
\end{equation*}Then $\Eb N={\frac{\log(2p)}{\log\Lambda_q(1)}}\ge1.$ Assume first that $\Eb N\ge8$. Set
\begin{equation*}
    \Eb Y_{ij}={\frac{\log\Lambda_q(1)^{1-2/q}}{2\Lambda_q(1)}},\quad X_i(j)=
B\left({\frac{\log(2p)}{n\log\Lambda_q(1)}}\right)^{-1/q}
R_i(Y_{ij}-\Eb Y_{ij}).
\end{equation*}
The envelope moment is at most $B^q$, because
\begin{equation*}
    \Eb\left|B\left({\frac{\log(2p)}{n\log\Lambda_q(1)}}\right)^{-1/q}R_i\right|^q=B^q.
\end{equation*}Also, the variance is at most
\begin{equation*}
    B^2\left({\frac{\log(2p)}{n\log\Lambda_q(1)}}\right)^{1-2/q} {\frac{\log\Lambda_q(1)^{1-2/q}}{2\Lambda_q(1)}} ={\frac{1}{2}}B^2\left({\frac{\log(2p)}{n}}\right)^{1-2/q}/\Lambda_q(1)={\frac{1}{2}}(\sigma\wedge B)^2.
\end{equation*}
We shall use the following consequence of Theorem~1 of \cite{zubkov2012full}: if $N$ is binomial and
$\Eb N\ge8$, then
\begin{equation}\label{eq:ZS_numeric}
\Pb\left(
\left\lceil {\frac{2\Eb N}{5}}\right\rceil
\le N\le
\left\lfloor {\frac{4\Eb N}{3}}\right\rfloor
\right)>0.63.
\end{equation} To prove \eqref{eq:ZS_numeric}, write $N\sim{\rm Bin}(n,p)$ and $\Eb N=np$. \cite{zubkov2012full} implies that, for $0<a<1$,
\begin{equation*}
    \Pb(N\le a\Eb N)
    \le
    \Phi\left(-\sqrt{2\Eb N\{a\log a-a+1\}}\right),
\end{equation*}
and, for $a>1$,
\begin{equation*}
    \Pb(N\ge a\Eb N)
    \le
    \Phi\left(-\sqrt{2\Eb N\{a\log a-a+1\}}\right).
\end{equation*}
Since $\Eb N\ge8$, $\lceil {2\Eb N/5}\rceil-1< {2\Eb N/5}$ and hence
\begin{equation*}
    \Pb\left(N<
    \left\lceil {\frac{2\Eb N}{5}}\right\rceil\right)
    \le
    \Pb\left(N\le {\frac{2\Eb N}{5}}\right)\le
    \Phi\left(
    -\sqrt{16\left\{
    {\frac{2}{5}}\log{\frac{2}{5}}+{\frac{3}{5}}
    \right\}}
    \right).
\end{equation*}
Similarly, if $\lfloor4\Eb N/3\rfloor\ge n$, then the upper tail is zero.
Otherwise, $\lfloor {4\Eb N/3}\rfloor+1> {4\Eb N/3},$ and so
\begin{equation*}
    \Pb\left(N>
    \left\lfloor {\frac{4\Eb N}{3}}\right\rfloor\right)
    \le
    \Pb\left(N\ge {\frac{4\Eb N}{3}}\right)\le
    \Phi\left(
    -\sqrt{16\left\{
    {\frac{4}{3}}\log{\frac{4}{3}}-{\frac{1}{3}}
    \right\}}
    \right).
\end{equation*}
Therefore,
\begin{align*}
    \Pb\left(
    \left\lceil {\frac{2\Eb N}{5}}\right\rceil
    \le N\le
    \left\lfloor {\frac{4\Eb N}{3}}\right\rfloor
    \right)
    \ge 1-\Phi\left(
    -\sqrt{16\left\{
    {\frac{4}{3}}\log{\frac{4}{3}}-{\frac{1}{3}}
    \right\}}
    \right)\\
    -\Phi\left(
    -\sqrt{16\left\{
    {\frac{4}{3}}\log{\frac{4}{3}}-{\frac{1}{3}}
    \right\}}
    \right)
    >0.63.
\end{align*}

On this event set $r=\lfloor2\Eb N/5\rfloor$. Then $r\le N$ and
$r\ge3\Eb N/10$. Conditionally on the $R_i$'s,
\begin{align*}
p\Pb(Z_j\ge r\mid R_1,\ldots,R_n)
&\ge
m(\Eb Y_{ij})^r(1-\Eb Y_{ij})^N.
\end{align*}
Since $\Eb Y_{ij}\le1/(2e)$, $r\le2\Eb N/5$, $N\le4\Eb N/3$, and $\log(p)=\Eb N\log\Lambda_q(1)-\log2$, we have
\begin{align*}
&\log\left\{
p(\Eb Y_{ij})^r(1-\Eb Y_{ij})^N
\right\} \\
&\qquad\ge
\Eb N\log\Lambda_q(1)-\log2
-{\frac{2\Eb N}{5}}(\log\Lambda_q(1)+\log2)
+{\frac{4\Eb N}{3}}\log\left(1-{\frac{1}{2e}}\right) \\
&\qquad\ge
8\left\{
{\frac{3}{5}}
-{\frac{2\log2}{5}}
+{\frac{4}{3}}\log\left(1-{\frac{1}{2e}}\right)
\right\}
-\log2
>
-\log2.
\end{align*}
Therefore $p\Pb(Z_j\ge r\mid R_1,\ldots,R_n)\ge {\frac{1}{2}},$ and so the conditional probability that $\max_j Z_j\ge r$ is at least
$1-e^{-1/2}$. On the intersection of this event with the event in
\eqref{eq:ZS_numeric},
\begin{align*}
\max_{1\le j\le m}\left|{\frac{1}{n}}\sum_{i=1}^nX_i(j)\right|
&\ge
{\frac{B}{n}}\left({\frac{\log(2p)}{n\log\Lambda_q(1)}}\right)^{-1/q}
r\left(1-{\frac{N\Eb Y_{ij}}{r}}\right)\\
&\ge
B\left({\frac{\log(2p)}{n\log\Lambda_q(1)}}\right)^{1-1/q}
{\frac{3}{10}}
\left(1-{\frac{20}{9e}}\right).
\end{align*}
Consequently,
\begin{equation*}
\Eb\left\Vert {\frac{1}{n}}\sum_{i=1}^nX_i\right\Vert_\infty\ge
0.63(1-e^{-1/2})
{\frac{3}{10}}
\left(1-{\frac{20}{9e}}\right)
B\left({\frac{\log(2p)}{n\log\Lambda_q(1)}}\right)^{1-1/q} \ge
{\frac{1}{74}}
B\left({\frac{\log(2p)}{n\log\Lambda_q(1)}}\right)^{1-1/q}.
\end{equation*}

Assume now that $1\le\Eb N<8$. Set
\begin{equation*}
    \Eb Y_{ij}={\frac{\log\Lambda_q(1)^{1-2/q}}{12\Lambda_q(1)}}
\end{equation*} and use the same definition of $X_i(j)$. The envelope moment is again at most
$B^q$, and the variance is at most $(\sigma\wedge B)^2/12$. Also,
\begin{align*}
\Pb(1\le N\le16)
&\ge
1-e^{-\Eb N}
-
\exp\left\{
-17\log\left({\frac{17}{\Eb N}}\right)+17-\Eb N
\right\} \\
&\ge
1-e^{-1}
-
\exp\left\{
-17\log\left({\frac{17}{8}}\right)+9
\right\}.
\end{align*}
Conditionally on the $R_i$'s, whenever $N\ge1$, $\Pb(Z_j\ge1\mid R_1,\ldots,R_n)\ge \Eb Y_{ij}.$ Since $\log(2p)\ge\log\Lambda_q(1)$,
\begin{equation*}
    p\Eb Y_{ij}={\frac{1}{24}} \exp\{\log(2p)-\log\Lambda_q(1)\}\log\Lambda_q(1)^{1-2/q}
\ge {\frac{1}{24}}.
\end{equation*} Thus the conditional probability that some $Z_j$ is at least one is at least
$1-e^{-1/24}$. On the event $1\le N\le16$ and $\max_jZ_j\ge1$,
\begin{equation*}
\max_{1\le j\le m}\left|{\frac{1}{n}}\sum_{i=1}^nX_i(j)\right|
\ge
{\frac{B}{n}}
\left({\frac{\log(2p)}{n\log\Lambda_q(1)}}\right)^{-1/q}
\left(1-N\Eb Y_{ij}\right) \ge
{\frac{B}{8}}
\left({\frac{\log(2p)}{n\log\Lambda_q(1)}}\right)^{1-1/q}
\left(1-{\frac{4}{3e}}\right),
\end{equation*}
where we used $\Eb N<8$ and $\Eb Y_{ij}\le1/(12e)$. Therefore
\begin{align*}
\Eb\left\Vert {\frac{1}{n}}\sum_{i=1}^nX_i\right\Vert_\infty\ge
\left[
1-e^{-1}
-
\exp\left\{
-17\log\left({\frac{17}{8}}\right)+9
\right\}
\right]
(1-e^{-1/24})
{\frac{1}{8}}
\left(1-{\frac{4}{3e}}\right)\\
\times
B\left({\frac{\log(2p)}{n\log\Lambda_q(1)}}\right)^{1-1/q}\ge
{\frac{1}{640}}
B\left({\frac{\log(2p)}{n\log\Lambda_q(1)}}\right)^{1-1/q}.
\end{align*}
Combining the cases proves \eqref{eq:maximal_lower_large}.
\end{proof}

\begin{proof}[Proof of Corollary~\ref{cor:upper}]
Let $P^n\in\Pc_{n,p}(q,\sigma,B)$ and put $S_n=n^{-1}\sum_iX_i$. Lemma \ref{lem:order} gives
\begin{equation}\label{eq:upper_start}
    \Eb\norm{S_n}_{(k),2}
    \le
    (1-e^{-1})^{-1}\sum_{s=0}^{\ceil{\log_2 k}}
    \sqrt{m_s}\,\Eb\left[\Eb\left(\max_{j\in J_{m_s}}|S_n(j)|\mid J_{m_s}\right)\right],
\end{equation} where $m_s = 2^s\wedge k$ and $J_m$ is uniformly distributed over subsets of $[p]$ of cardinality $\ceil{p/m}$. For every realization of $J_{m_s}$, the subvectors $(X_i(j))_{j\in J_{m_s}}$ satisfy the same variance and envelope constraints in dimension $\ceil{p/m_s}$. Since
\begin{equation*}
    \log(2p/m_s)\le \log(2\ceil{p/m_s})\le \log(4p/m_s)\le2\log(2p/m_s),
\end{equation*}
Proposition \ref{prop:sup_upper} and Lemma \ref{lem:sum} give
\begin{equation*}
    \Eb\norm{S_n}_{(k),2}
    \le 8\sqrt{k}B
\end{equation*}
and
\begin{align*}
    \Eb\norm{S_n}_{(k),2}
    &\le
    90\sqrt{k}(\sigma\wedge B)\sqrt{\frac{\log(2p/k)}{n}}
    +187\sqrt{k}B\left({\frac{\log(2p/k)}{n}}\right)^{1-1/q}.
\end{align*}
If $\log\Lambda_q(k)\le1$, then
\begin{equation*}
    B\left({\frac{\log(2p/k)}{n}}\right)^{1-1/q}
    =(\sigma\wedge B)\sqrt{\frac{\log(2p/k)}{n}}\,\Lambda_q(k)^{1/2}
    \le e^{1/2}(\sigma\wedge B)\sqrt{\frac{\log(2p/k)}{n}}.
\end{equation*}

Assume now that $\log\Lambda_q(k)>1$. Put
\begin{equation*}
    \Delta_s
    =
    \log\left\{
        {\frac{B^2}{(\sigma\wedge B)^2}}
        \left({\frac{\log(2p/m_s)}{n}}\right)^{1-2/q}
    \right\},
    \quad
    \Delta_s^\circ
    =
    \log\left\{
        {\frac{B^2}{(\sigma\wedge B)^2}}
        \left({\frac{\log(2\ceil{p/m_s})}{n}}\right)^{1-2/q}
    \right\}.
\end{equation*}
Since $m_s\leq k$ and $1-2/q\geq0$, we have $\log(2p/m_s)\geq\log(2p/k),$ and hence
\begin{equation*}
    \Delta_s
    \geq
    \log\left\{
        {\frac{B^2}{(\sigma\wedge B)^2}}
        \left({\frac{\log(2p/k)}{n}}\right)^{1-2/q}
    \right\}
    =
    \log\Lambda_q(k)>1.
\end{equation*}
Moreover, since $\ceil{p/m_s}\geq p/m_s$, we have $\Delta_s^\circ\geq\Delta_s>1.$ Thus the bound in \eqref{eq:sup_upper_log} applies to
each conditional subvector appearing in \eqref{eq:upper_start}. For the comparison of the summands, note also that
\begin{equation*}
    \log(2\ceil{p/m_s})
    \leq
    \log(2p/m_s+2)
    \leq
    \log(4p/m_s)
    \leq
    2\log(2p/m_s)
    =
    2\log(2p/m_s).
\end{equation*}
Since $\Delta_s^\circ\geq\Delta_s$, it follows that
\begin{equation*}
    \left(
        {\frac{\log(2\ceil{p/m_s})}{n\Delta_s^\circ}}
    \right)^{1-1/q}\leq 2^{1-1/q}
    \left(
        {\frac{\log(2p/m_s)}{n\Delta_s}}
    \right)^{1-1/q}\leq 2
    \left(
        {\frac{\log(2p/m_s)}{n\Delta_s}}
    \right)^{1-1/q}.
\end{equation*} Applying \eqref{eq:sup_upper_log} in \eqref{eq:upper_start} gives
\begin{align*}
    \Eb\norm{S_n}_{(k),2}
    &\leq
    (1-e^{-1})^{-1}
    \sum_{s=0}^S
    \sqrt{m_s}\,
    8B
    \left(
        {\frac{\log(2\ceil{p/m_s})}{n\Delta_s^\circ}}
    \right)^{1-1/q} \\
    &\leq
    16(1-e^{-1})^{-1}B
    \sum_{s=0}^S
    \sqrt{m_s}
    \left(
        {\frac{\log(2p/m_s)}{n\Delta_s}}
    \right)^{1-1/q}.
\end{align*}
By \eqref{eq:sum_log_delta},
\begin{equation*}
    \sum_{s=0}^S
    \sqrt{m_s}
    \left(
        {\frac{\log(2p/m_s)}{n\Delta_s}}
    \right)^{1-1/q}
    \leq
    (9+5\sqrt{2})\sqrt{k}
    \left(
        {\frac{\log(2p/k)}{n\log\Lambda_q(k)}}
    \right)^{1-1/q}.
\end{equation*}
Hence, we conclude that
\begin{equation*}
    \Eb\norm{S_n}_{(k),2}\leq 
    16 (1-e^{-1})^{-1}(9+5\sqrt{2})
    \sqrt{k}B
    \left(
        {\frac{\log(2p/k)}{n\log\Lambda_q(k)}}
    \right)^{1-1/q} \leq
    407\sqrt{k}B
    \left(
        {\frac{\log(2p/k)}{n\log\Lambda_q(k)}}
    \right)^{1-1/q}.
\end{equation*}
Together with the bound $\Eb\norm{S_n}_{(k),2}\le\sqrt{k}B$ and taking the supremum over $P^n$ completes the proof of Corollary~\ref{cor:upper}.
\end{proof}

\begin{proof}[Proof of Corollary~\ref{cor:lower}]
Let $p_k=\ceil{p/k}$. If $\log\Lambda_q(k)\le1$, Theorem~\ref{thm:maximal_lower} and Lemma \ref{lem:block} give
\begin{align*}
    &\Ec^*_{n,p,k,q}(\sigma,B)
    \ge {\sqrt{\frac{k}{2}}}\Ec^*_{n,p_k,1,q}(\sigma,B) \\
    &\ge {\frac{\sqrt{k}}{8\sqrt{2}}}
    \min\left\{B,
    (\sigma\wedge B)\sqrt{\frac{\log(2p_k)}{n}}\right\}\ge {\frac{\sqrt{k}}{8\sqrt{2}}}
    \min\left\{B,
    (\sigma\wedge B)\sqrt{\frac{\log(2p/k)}{n}}\right\}.
\end{align*}
This proves \eqref{eq:main_lower_small} after decreasing the constant.

Assume next that $\log\Lambda_q(k)>1$ and $\log(2p/k)\geq \log\Lambda_q(k)$ holds. Hence, an application of Theorem~\ref{thm:maximal_lower} with the dimension $p_k$, combined with Lemma~\ref{lem:block}, gives
\begin{align*}
    \Ec^*_{n,p,k,q}(\sigma,B)
    &\ge {\sqrt{\frac{k}{2}}}\Ec^*_{n,p_k,1,q}(\sigma,B)\\
    &\ge {\frac{\sqrt{k}}{640\sqrt{2}}}
    \min\left\{B,
    B\left(
    {\frac{\log(2p_k)}{n\log\left\{{\frac{B^2}{(\sigma\wedge B)^2}}
    (\log(2p_k)/n)^{1-2/q}\right\}}}
    \right)^{1-1/q}\right\}  \\
    &\ge {\frac{\sqrt{k}}{1280\sqrt{2}}}
    \min\left\{B,
    B\left({\frac{\log(2p/k)}{n\log\Lambda_q(k)}}\right)^{1-1/q}\right\},
\end{align*}
where the second inequality follows from \eqref{eq:scale_preserved}. Applying \eqref{eq:block_amp} proves \eqref{eq:main_lower_large} after decreasing the constant, and completes the proof of Corollary~\ref{cor:lower}.
\end{proof}

\begin{proof}[Proof of Proposition \ref{prop:sharpness}]
It suffices to take $n=1$, $B=1$, and $k=1$. For any admissible law with $\sigma\le1$,
\begin{equation}\label{eq:sharp_upper}
    \Eb\norm{X}_\infty
    \le \min\left\{1,\left(\Eb\sum_{j=1}^pX(j)^2\right)^{1/2}\right\}
    \le \min\{1,\sigma\sqrt p\}.
\end{equation}
Choose $c_1\in(c,1)$ and integers $p_r$ such that $\log(2p_r)/r\to c_1$. Put
\begin{equation*}
    \sigma_r^2=\{\log(2p_r)\}^{1-2/q}e^{-r},
\end{equation*}
where the power is zero when $q=2$. Then $\Lambda_q(1)=e^r$ and $\log(2p_r)\ge c\log\Lambda_q(1)$ for all sufficiently large $r$. By \eqref{eq:sharp_upper},
\begin{equation*}
    \Ec^*_{1,p_r,1,q}(\sigma_r,1)
    \le \sigma_r\sqrt{p_r} \le
    \{\log(2p_r)\}^{(1-2/q)/2}
    \exp\left\{ {\frac{\log(2p_r)-r}{2}}\right\}
    \to0.
\end{equation*}
On the other hand,
\begin{equation*}
    \min\left\{1,
    \left({\frac{\log(2p_r)}{\log\Lambda_q(1)}}\right)^{1-1/q}\right\}
    \to c_1^{1-1/q}>0.
\end{equation*}
This contradiction proves the proposition.
\end{proof}

\subsection{Proofs for Section~\ref{sec:envelope-subweibull}}\label{app:proof-sub-orlicz}
\begin{proof}[proof of Theorem~\ref{thm:sub_orlicz_maximal} and Corollary~\ref{cor:sub_orlicz}]
We prove the top-$k$ statement in Corollary~\ref{cor:sub_orlicz}; Theorem~\ref{thm:sub_orlicz_maximal} is the case $k=1$. We first prove the upper bound. Let $P^n\in\Pc_{n,p}^{\psi_\alpha,\infty}(\sigma,K)$. For
$X_i\sim P_i$, put $M_i=\max_{1\leq j\leq p}|X_i(j)|.$ Then $\Pb(M_i>Kt)\leq2e^{-t^\alpha}$ for $t\geq 0$. Consequently,
\begin{align}
    \Eb M_i
    &\leq
    K+2K\Gamma(1+1/\alpha),                                      \label{eq:sub_orlicz_first_moment}\\
    \Eb M_i^2
    &\leq
    K^2+4K^2\Gamma(1+2/\alpha).                                  \label{eq:sub_orlicz_second_moment}
\end{align}
Since $\Gamma(1+2/\alpha)\geq1/2$ and
$\Gamma(1+1/\alpha)^2\leq\Gamma(1+2/\alpha)$,
\eqref{eq:sub_orlicz_first_moment} gives
\begin{equation}\label{eq:sub_orlicz_trivial_k}
    \Eb\left\|\frac1n\sum_{i=1}^nX_i\right\|_{(k),2}
    \leq
    \sqrt{k}\{K+2K\Gamma(1+1/\alpha)\}
    \leq
    4\sqrt{\Gamma(1+2/\alpha)}\sqrt{k}K.
\end{equation}
We record the corresponding sup-norm bound in
dimension $m$. Suppose $X_i\in\Real^m$, $\Eb X_i=0_m$,
\begin{equation*}
    \max_{1\leq j\leq m}
    \frac1n\sum_{i=1}^n\Eb X_i(j)^2\leq \sigma^2,
    \qquad
    \max_{1\leq i\leq n}
    \norm{\max_{1\leq j\leq m}|X_i(j)|}_{\psi_\alpha}\leq K .
\end{equation*}
If $\log(2p)>n$, then \eqref{eq:sub_orlicz_trivial_k} with $k=1$ is enough. Assume
$\log(2p)\leq n$ and set
\begin{equation*}
    \tau
    =
    K\left[
    \log\left(e+\frac{n}{\log(2p)}\right)
    \right]^{1/\alpha}.
\end{equation*}
For each coordinate, decompose 
\begin{equation*}
    X_i(j) = \big(X_i(j)\mathbf 1\{M_i\leq\tau\}-\Eb X_i(j)\mathbf 1\{M_i\leq\tau\}\big) + \big(X_i(j)\mathbf 1\{M_i>\tau\} - \Eb X_i(j)\mathbf 1\{M_i>\tau\}\big).
\end{equation*} The absolute value of the former term is bounded by $2\tau$, and by
\eqref{eq:sub_orlicz_second_moment} its coordinate-wise variance is at most $\sigma^2\wedge \{K^2+4K^2\Gamma(1+2/\alpha)\}.$ Then Bennett's inequality in Lemma~\ref{lem:bennett}, followed by a union bound over the $2m$ signed coordinates and integration in the tail parameter, gives
\begin{align}
    &\Eb\left\|
    \frac1n\sum_{i=1}^n
    \left\{
    X_i\mathbf 1\{M_i\leq\tau\}
    -
    \Eb X_i\mathbf 1\{M_i\leq\tau\}
    \right\}
    \right\|_\infty                                                        \notag\\
    &\leq
    \frac52
    \left[
    \sigma\wedge K\{1+4\Gamma(1+2/\alpha)\}^{1/2}
    \right]
    \sqrt{\frac{\log(2p)}{n}} +
    \frac53K\frac{\log(2p)}{n}
    \left[
    \log\left(e+\frac{n}{\log(2p)}\right)
    \right]^{1/\alpha}.                                           \label{eq:sub_orlicz_bounded_part}
\end{align}
For the tail part,
\begin{align*}
    \Eb\left[M_i\mathbf 1\{M_i>\tau\}\right]
    &\leq
    2\tau\exp\left\{-\left(\frac{\tau}{K}\right)^\alpha\right\}
    +
    \frac{2K}{\alpha}
    \int_{(\tau/K)^\alpha}^\infty s^{1/\alpha-1}e^{-s}\,ds .
\end{align*}
For $x\geq1$,
\begin{equation*}
    \int_x^\infty s^{1/\alpha-1}e^{-s}\,ds
    \leq
    x^{1/\alpha}e^{-x}
    \int_0^\infty(1+y)^{1/\alpha-1}e^{-y}\,dy .
\end{equation*}
Since $\log(2p)\leq n$,
\begin{equation*}
    \exp\left\{
    -\log\left(e+\frac{n}{\log(2p)}\right)
    \right\}
    \leq
    \frac{\log(2p)}{n}.
\end{equation*}
Hence
\begin{align}
    \Eb\left[M_i\mathbf 1\{M_i>\tau\}\right]
    &\leq
    \left[
    2+
    \frac{2}{\alpha}
    \int_0^\infty(1+y)^{1/\alpha-1}e^{-y}\,dy
    \right]
    K\frac{\log(2p)}{n}
    \left[
    \log\left(e+\frac{n}{\log(2p)}\right)
    \right]^{1/\alpha}.                                           \label{eq:sub_orlicz_tail_part}
\end{align}
If $\alpha\leq1$, then
\begin{equation*}
    \frac1\alpha
    \int_0^\infty(1+y)^{1/\alpha-1}e^{-y}\,dy
    \leq
    e\Gamma(1+1/\alpha)
    \leq
    e\sqrt{\Gamma(1+2/\alpha)}.
\end{equation*}
If $\alpha>1$, the same term is at most $1$, and this is bounded by
$\sqrt2\sqrt{\Gamma(1+2/\alpha)}$. Thus the right hand side of
\eqref{eq:sub_orlicz_tail_part} is bounded by
\begin{equation*}
    8e\sqrt{\Gamma(1+2/\alpha)}
    K\frac{\log(2p)}{n}
    \left[
    \log\left(e+\frac{n}{\log(2p)}\right)
    \right]^{1/\alpha}.
\end{equation*}
Combining this with \eqref{eq:sub_orlicz_bounded_part}, and using
\begin{equation*}
    \{1+4\Gamma(1+2/\alpha)\}^{1/2}
    \leq
    \sqrt6\sqrt{\Gamma(1+2/\alpha)},
\end{equation*}
gives, for all $m\geq1$,
\begin{align}
    \Eb\left\|
    \frac1n\sum_{i=1}^nX_i
    \right\|_\infty
    &\leq
    60\sqrt{\Gamma(1+2/\alpha)}
    \min\Bigg\{
    K,
    \max\bigg\{
    (\sigma\wedge K)\sqrt{\frac{\log(2p)}{n}},                     \notag\\
    &\hspace{4.2cm}
    K\frac{\log(2p)}{n}
    \left[
    \log\left(e+\frac{n}{\log(2p)}\right)
    \right]^{1/\alpha}
    \bigg\}
    \Bigg\}.
    \label{eq:sub_orlicz_sup_bound}
\end{align}

We now pass from the sup-norm to $\norm{\cdot}_{(k),2}$. Lemma~\ref{lem:order} and
\eqref{eq:sub_orlicz_sup_bound}, applied conditionally to the random coordinate blocks,
give
\begin{align*}
    &\Eb\left\|
    \frac1n\sum_{i=1}^nX_i
    \right\|_{(k),2}
    \leq
    60(1-e^{-1})^{-1}\sqrt{\Gamma(1+2/\alpha)}
    \sum_{s=0}^{\ceil{\log_2 k}}
    \sqrt{m_s}\times
    \min\Bigg\{
    K,\\
    &\max\bigg\{
    (\sigma\wedge K)\sqrt{\frac{\log(2\ceil{p/m_s})}{n}},
    K\frac{\log(2\ceil{p/m_s})}{n}
    \left[
    \log\left(e+\frac{n}{\log(2\ceil{p/m_s})}\right)
    \right]^{1/\alpha}
    \bigg\}
    \Bigg\},
\end{align*}
where $m_s=2^s\wedge k$. Combining the preceding display with the upper bound in
\eqref{eq:sub_orlicz_trivial_k}, and using
\begin{equation*}
    \log(2p/m_s)
    \leq
    \log(2\ceil{p/m_s})
    \leq
    2\log(2p/m_s),
\end{equation*}
together with the summation bounds in Lemma~\ref{lem:sum}, gives
\begin{align*}
    \Eb\left\|
    \frac1n\sum_{i=1}^nX_i
    \right\|_{(k),2}
    &\leq
    4000\sqrt{\Gamma(1+2/\alpha)}\sqrt{k}
    \min\Bigg\{
    K,
    \max\bigg\{
    (\sigma\wedge K)\sqrt{\frac{\log(2p/k)}{n}},\\
    &\hspace{4.2cm}
    K\frac{\log(2p/k)}{n}
    \left[
    \log\left(e+\frac{n}{\log(2p/k)}\right)
    \right]^{1/\alpha}
    \bigg\}
    \Bigg\}.
\end{align*}
This proves \eqref{eq:sub_orlicz_upper_readable}.

We prove the lower bound. The variance term follows from the Rademacher construction.
Let
\begin{equation*}
    X_i(j)=(\sigma\wedge K)(\log2)^{1/\alpha}\varepsilon_{ij},
    \qquad 1\leq i\leq n,
    \quad 1\leq j\leq p,
\end{equation*}
where the $\varepsilon_{ij}$'s are independent Rademacher variables. Then
$\max_j\Eb X_i(j)^2\leq\sigma^2$ and $\norm{\max_{1\leq j\leq p}|X_i(j)|}_{\psi_\alpha}\leq K.$ Lemma~\ref{lem:block}, together with the $k=1$ Rademacher lower bound, gives
\begin{align}
    \Ec_{n,p,k,\psi_\alpha}^{\infty,*}(\sigma,K)
    &\geq
    \frac{(\log2)^{1/\alpha}\sqrt{k}}{8\sqrt2}
    \min\left\{
    K,
    (\sigma\wedge K)\sqrt{\frac{\log(2p/k)}{n}}
    \right\}\nonumber\\
    &\geq \frac{2^{-1/\alpha}\sqrt{k}}{8\sqrt2}
    \min\left\{
    K,
    (\sigma\wedge K)\sqrt{\frac{\log(2p/k)}{n}}
    \right\}.
    \label{eq:sub_orlicz_variance_lower_sparse}
\end{align}
Meanwhile, by Lemma~\ref{lem:block}, it is enough to
construct an iid law in dimension $\ceil{p/k}$; this changes the final constant by at most
the factor $\sqrt{k}/\sqrt2$. We give the construction for $k=1$. Fix
$\Delta\in\Dc_{\alpha,1}$ and put
\begin{equation*}
    \eta=\frac{\log(2p)}{n\Delta},
    \qquad
    A=K\{\log(1+\eta^{-1})\}^{1/\alpha},
    \qquad
    \rho=\frac12e^{-\Delta}.
\end{equation*}
The definition of $\Dc_{\alpha,1}$ implies $0<\eta\leq1$ and $n\eta\geq1$. Let
$R_1,\ldots,R_n$ be iid Bernoulli random variables with success probability $\eta$, and
let $Y_{ij}$, $1\leq i\leq n$, $1\leq j\leq p$, be iid Bernoulli random variables with
success probability $\rho$, independent of the $R_i$'s. Define
\begin{equation*}
    X_i(j)=AR_i(Y_{ij}-\rho).
\end{equation*}
Then $\Eb X_i(j)=0$ and $ \max_{1\leq j\leq p}|X_i(j)|\leq AR_i$. Therefore
\begin{align*}
    \Eb\exp\left[
    \left(\frac{\max_{1\leq j\leq p}|X_i(j)|}{K}\right)^\alpha
    \right]
    &\leq
    1-\eta+
    \eta\exp\{(A/K)^\alpha\}  \\
    &=
    1-\eta+
    \eta(1+\eta^{-1})
    =2.
\end{align*}
Thus, $\norm{\max_{1\leq j\leq p}|X_i(j)|}_{\psi_\alpha}\leq K$ and 
\begin{equation*}
    \max_{1\leq j\leq p}\Eb X_i(j)^2=A^2\eta\rho(1-\rho)\leq \frac12K^2\frac{\log(2p)}{n\Delta}\left[\log\left(e+\frac{n\Delta}{\log(2p)}\right)\right]^{2/\alpha}e^{-\Delta}\leq\sigma^2.
\end{equation*} Hence the law of $X$ belongs to the class $\Pc_{1,p}^{\psi_\alpha,\infty}(\sigma,K)$. Let $    N=\sum_{i=1}^nR_i.$ Conditional on $R_1,\ldots,R_n$, the random variables
\begin{equation*}
    \sum_{i=1}^nR_iY_{ij},
    \qquad 1\leq j\leq p,
\end{equation*}
are iid binomial random variables with parameters $N$ and $\rho$.

First suppose that $n\eta\geq16$. Chernoff's inequalities \citep{Chernoff1952} give
\begin{equation*}
    \Pb\left(\frac12n\eta\leq N\leq \frac43n\eta\right)
    \geq
    1-e^{-2}-e^{-16/21}.
\end{equation*}
On this event, set $m=\ceil{3n\eta/10}$. Then $m\leq29n\eta/80$ and $m\leq N$.
For $Z\sim\operatorname{Binomial}(N,\rho)$,
\begin{equation*}
    \Pb(Z\geq m)\geq \Pb(Z=m)\geq \rho^m(1-\rho)^N.
\end{equation*}
Using $\rho=e^{-\Delta}/2$, $\Delta\geq1$, and
$\log(1-x)\geq -x/(1-x)$, we obtain
\begin{align*}
    \log\{p\Pb(Z\geq m)\}
    &\geq
    \log(p)
    -
    m\Delta
    -
    m\log2
    +
    N\log(1-\rho)\\
    &\geq
    \log(p)
    -
    \frac{29}{80}\frac{\log(2p)}{\Delta}(\Delta+\log2)
    -
    \frac{2}{3(e-1/2)}\frac{\log(2p)}{\Delta}\geq -1,
\end{align*} where the last inequality follows from $\Delta\geq1$ and
\begin{equation*}
    \frac{29\log2}{80}+\frac{2}{3(e-1/2)}<\frac{51}{80}.
\end{equation*} Hence $p\Pb(Z\geq m)\geq e^{-1}$, and therefore, 
\begin{equation*}
    \Pb\left(
    \max_{1\leq j\leq p}\sum_{i=1}^nR_iY_{ij}\geq m
    \,\middle|\,
    R_1,\ldots,R_n
    \right)
    \geq
    1-e^{-e^{-1}}.
\end{equation*}
On this event,
\begin{equation*}
    \left\|
    \frac1n\sum_{i=1}^nX_i
    \right\|_\infty
    \geq
    \frac{A}{n}(m-N\rho)
    \geq
    A\eta\left(\frac3{10}-\frac{2}{3e}\right).
\end{equation*}
Thus, in the case $n\eta\geq16$,
\begin{align*}
    \Eb\left\|
    \frac1n\sum_{i=1}^nX_i
    \right\|_\infty
    &\geq
    (1-e^{-2}-e^{-16/21})(1-e^{-e^{-1}})
    \left(\frac3{10}-\frac{2}{3e}\right)A\eta.
\end{align*}

It remains to consider $1\leq n\eta<16$. Since $\Delta\leq\log(2p)$,
\begin{equation*}
    p\rho=\frac p2e^{-\Delta}\geq \frac14.
\end{equation*}
Also, $N$ is binomial with mean $n\eta\in[1,16)$. Chernoff's inequality gives
\begin{equation*}
    \Pb(1\leq N\leq32)
    \geq
    1-e^{-1}-e^{16-32\log2}.
\end{equation*}
Condition on $R_1,\ldots,R_n$ and assume $1\leq N\leq32$. If $N\rho\leq1/2$, then
\begin{equation*}
    \Pb\left(
    \max_{1\leq j\leq p}\sum_{i=1}^nR_iY_{ij}\geq1
    \,\middle|\,
    R_1,\ldots,R_n
    \right)
    =
    1-(1-\rho)^p
    \geq
    1-e^{-1/4},
\end{equation*}
and on this event
\begin{equation*}
    \left\|
    \frac1n\sum_{i=1}^nX_i
    \right\|_\infty
    \geq
    \frac{A}{n}(1-N\rho)
    \geq
    \frac{A}{2n}.
\end{equation*}
If $N\rho>1/2$, then
\begin{equation*}
    \Pb\left(
    \sum_{i=1}^nR_iY_{i1}=0
    \,\middle|\,
    R_1,\ldots,R_n
    \right)
    =
    (1-\rho)^N
    \geq
    \left(1-\frac1{2e}\right)^{32},
\end{equation*}
and on this event
\begin{equation*}
    \left\|
    \frac1n\sum_{i=1}^nX_i
    \right\|_\infty
    \geq
    \frac{AN\rho}{n}
    \geq
    \frac{A}{2n}.
\end{equation*}
Since $n\eta<16$, this gives
\begin{align*}
    \Eb\left\|
    \frac1n\sum_{i=1}^nX_i
    \right\|_\infty
    &\geq
    \frac{1-e^{-1}-e^{16-32\log2}}{32}
    \min\left\{
    1-e^{-1/4},
    \left(1-\frac1{2e}\right)^{32}
    \right\}A\eta.
\end{align*}
Combining the two cases and then applying Lemma~\ref{lem:block}, one has
\begin{equation*}
    \Eb\left\|
    \frac1n\sum_{i=1}^nX_i
    \right\|_{(k),2}\gtrsim\sqrt{k}\,
    K\frac{\log(2p/k)}{n\Delta}
    \left[
    \log\left(e+\frac{n\Delta}{\log(2p/k)}\right)
    \right]^{1/\alpha}.
\end{equation*}Taking the supremum over $\Delta\in\Dc_{\alpha,k}$ and combining with
\eqref{eq:sub_orlicz_variance_lower_sparse} proves
\eqref{eq:sub_orlicz_lower_readable}.
\end{proof}

\begin{proof}[proof of \eqref{eq:optimization}]
    By Stirling's formula \citep{Robbins1955},
\begin{equation}\label{eq:stirling}
    0<
    \inf_{q\geq2}
    \frac{
        \left\{2\Gamma\left(1+q/\alpha\right)\right\}^{1/q}
    }{q^{1/\alpha}}
    \leq
    \sup_{q\geq2}
    \frac{
        \left\{2\Gamma\left(1+q/\alpha\right)\right\}^{1/q}
    }{q^{1/\alpha}}
    <\infty .
\end{equation}
Indeed, the displayed ratio is continuous on $[2,\infty)$ and converges to
$e^{-1/\alpha}\alpha^{-1/\alpha}$ as $q\to\infty$. Hence it is enough to
prove the corresponding estimate for
\begin{equation*}
    \inf_{q\geq2} q^{1/\alpha}x^{1-1/q}.
\end{equation*}

Write $s=\log(1/x)\geq0$, then $q^{1/\alpha}x^{1-1/q}=xq^{1/\alpha}e^{s/q}.$ For every $q\geq2$, one has $q^{1/\alpha}e^{s/q}\geq 2^{1/\alpha}$ and minimizing over all $q>0$ gives
\begin{equation*}
    \inf_{q>0}q^{1/\alpha}e^{s/q}=
    (e\alpha s)^{1/\alpha},
\end{equation*} where the infimium is attained by $q=s\alpha$. Therefore, for every $q\geq2$,
\begin{equation*}
    q^{1/\alpha}e^{s/q} \geq \max\left\{2^{1/\alpha},(e\alpha s)^{1/\alpha}\right\} \geq 2^{-1/\alpha}\min\{2,e\alpha\}^{1/\alpha}(1+s)^{1/\alpha}.
\end{equation*}
Since $1+s=\log(e/x)$, this yields
\begin{align}\label{eq:lower_q_gamma_proxy}
    \inf_{q\geq2}q^{1/\alpha}x^{1-1/q}
    \geq
    2^{-1/\alpha}\min\{2,e\alpha\}^{1/\alpha}
    x\{\log(e/x)\}^{1/\alpha}\nonumber\\
    \geq2^{-1/\alpha}\min\{2,e\alpha\}^{1/\alpha}\min\left\{
        x^{1/2},
        x\{\log(e/x)\}^{1/\alpha}
    \right\}.
\end{align}

For the upper bound, first take $q=2$. This gives
\begin{equation}\label{eq:upper_q_2}
    \inf_{q\geq2}q^{1/\alpha}x^{1-1/q}
    \leq
    2^{1/\alpha}x^{1/2}.
\end{equation}
Next, if $\alpha\log(e/x)\geq2$, take $q=\alpha\log(e/x)$. Since
$\log(1/x)\leq \log(e/x)$,
\begin{align}
    \inf_{q\geq2}q^{1/\alpha}x^{1-1/q}
    &\leq
    \{\alpha\log(e/x)\}^{1/\alpha}
    x\exp\left\{
        \frac{\log(1/x)}{\alpha\log(e/x)}
    \right\} \notag \\
    &\leq
    (e\alpha)^{1/\alpha}
    x\{\log(e/x)\}^{1/\alpha}.
    \label{eq:upper_interior}
\end{align}
If $\alpha\log(e/x)<2$, then taking $q=2$ gives
\begin{align}
    \inf_{q\geq2}q^{1/\alpha}x^{1-1/q}
    &\leq
    2^{1/\alpha}x^{1/2} \notag \\
    &=
    2^{1/\alpha}
    x\{\log(e/x)\}^{1/\alpha}
    \frac{x^{-1/2}}{\{\log(e/x)\}^{1/\alpha}} \notag \\
    &\leq
    2^{1/\alpha}
    \exp\left\{\frac{(2/\alpha-1)_+}{2}\right\}
    x\{\log(e/x)\}^{1/\alpha},
    \label{eq:upper_endpoint}
\end{align}
because $\log(e/x)\geq1$ and
$\log(1/x)=\log(e/x)-1<(2/\alpha-1)_+$ in this case. Combining
\eqref{eq:upper_q_2}, \eqref{eq:upper_interior}, and
\eqref{eq:upper_endpoint}, we obtain
\begin{equation*}
    \inf_{q\geq2}q^{1/\alpha}x^{1-1/q}
    \leq
    C'_\alpha
    \min\left\{
        x^{1/2},
        x\{\log(e/x)\}^{1/\alpha}
    \right\},
\end{equation*}
where one may take
\begin{equation*}
    C'_\alpha
    =
    \max\left\{
        2^{1/\alpha},
        (e\alpha)^{1/\alpha},
        2^{1/\alpha}
        \exp\left(\frac{(2/\alpha-1)_+}{2}\right)
    \right\}.
\end{equation*}
Combining the last display, \eqref{eq:lower_q_gamma_proxy}, and \eqref{eq:stirling} proves \eqref{eq:optimization}.
\end{proof}

\subsection{Proofs for Section~\ref{sec:coordinate-wise}}\label{app:proof-coordinate-wise}
\begin{proof}[Proof of Theorem~\ref{thm:k1}]

Put, within this proof only,
\begin{equation*}
    r=1\vee\log(p),
    \qquad
    A=\log\left(e+\frac{n}{r}\right).
\end{equation*}
Then $r\geq1$, $A\geq1$, $ne^{-A}\leq r$, and $p^{1/r}\leq e$.  The sub-Weibull assumption gives, for every $u\geq1$,
\begin{equation}
    \Pp\left\{\abs{X_i(j)}>K\{(1+\log2)u\}^{1/\alpha}\right\}
    \leq
    2e^{-(1+\log2)u}
    \leq e^{-u}.
    \label{eq:tail_calibrated}
\end{equation}

First consider
\begin{align*}
    Y_i(j)
    ={}&X_i(j)\one\{\abs{X_i(j)}\leq K\{(1+\log2)A\}^{1/\alpha}\}\\
    &-\Eb X_i(j)\one\{\abs{X_i(j)}\leq K\{(1+\log2)A\}^{1/\alpha}\}.
\end{align*}
These variables are centered, their absolute values are bounded by
\begin{equation*}
    2K\{(1+\log2)A\}^{1/\alpha},
\end{equation*}
and their variance sum in each coordinate is bounded by
\begin{equation*}
    n\left[\sigma^2\wedge 2\Gamma(1+2/\alpha)K^2\right].
\end{equation*}
The second bound follows from $\Eb\abs{X_i(j)}^2\leq 2\Gamma(1+2/\alpha)K^2$.  Lemma~\ref{lem:bernstein_integral} therefore gives
\begin{align}
    \Eb\max_{1\leq j\leq p}
    \abs{\frac{1}{n}\sum_{i=1}^nY_i(j)}
    &\leq
    \left\{\sigma\wedge \bigl(2\Gamma(1+2/\alpha)\bigr)^{1/2}K\right\}
    \sqrt\frac{2}{n}
    \left\{\sqrt{\log(2p)}+\frac{1}{2\sqrt{\log(2p)}}\right\}
    \notag\\
    &\quad+
    \frac{4(1+\log2)^{1/\alpha}K}{3n}
    \{\log(2p)+1\}A^{1/\alpha}.
    \label{eq:bounded_part}
\end{align}

For the tail part, set
\begin{equation*}
    U_i(j)=\abs{X_i(j)}
    \one\{\abs{X_i(j)}>K\{(1+\log2)A\}^{1/\alpha}\}.
\end{equation*}
For each fixed $j$, the variables $U_i(j)$ are stochastically dominated by
\begin{equation*}
    K(1+\log2)^{1/\alpha}\eta_i(A+\xi_i)^{1/\alpha},
\end{equation*}
where $\eta_i$ are independent Bernoulli variables with mean $e^{-A}$, $\xi_i$ are independent standard exponential variables, and the two families are independent.  Indeed, if $0\leq t<K\{(1+\log2)A\}^{1/\alpha}$, then \eqref{eq:tail_calibrated} gives
\begin{equation*}
    \Pp\{U_i(j)>t\}
    \leq
    \Pp\{\abs{X_i(j)}>K\{(1+\log2)A\}^{1/\alpha}\}
    \leq e^{-A},
\end{equation*}
which equals the probability that the dominating variable is positive.  If $t=K\{(1+\log2)u\}^{1/\alpha}$ with $u\geq A$, then
\begin{equation*}
    \Pp\{U_i(j)>t\}\leq e^{-u}
    =
    \Pp\left\{K(1+\log2)^{1/\alpha}\eta_i(A+\xi_i)^{1/\alpha}>t\right\}.
\end{equation*}

Assume first that $\alpha\geq1$.  By \eqref{eq:subadd},
\begin{equation*}
    \sum_{i=1}^n\eta_i(A+\xi_i)^{1/\alpha}
    \leq
    2A^{1/\alpha}\sum_{i=1}^n\eta_i
    +
    \sum_{i=1}^n\eta_i\xi_i .
\end{equation*}
Since $ne^{-A}\leq r$, Lemma~\ref{lem:bernoulli_moment} gives
\begin{equation*}
    \left\|\sum_{i=1}^n\eta_i\right\|_{L_r}\leq(e+1)r.
\end{equation*}
Conditioning on $\eta_1,\ldots,\eta_n$ and using Lemma~\ref{lem:gamma_moment} gives
\begin{equation*}
    \left\|\sum_{i=1}^n\eta_i\xi_i\right\|_{L_r}
    \leq
    \left\|\sum_{i=1}^n\eta_i+r\right\|_{L_r}
    \leq(e+2)r.
\end{equation*}
Therefore
\begin{equation*}
    \left\|\sum_{i=1}^n\eta_i(A+\xi_i)^{1/\alpha}\right\|_{L_r}
    \leq
    (3e+4)rA^{1/\alpha}.
\end{equation*}
Using $p^{1/r}\leq e$,
\begin{equation*}
    \Eb\max_{1\leq j\leq p}\sum_{i=1}^nU_i(j)
    \leq
    e(3e+4)(1+\log2)^{1/\alpha}KrA^{1/\alpha}.
\end{equation*}
Moreover,
\begin{align*}
    \Eb\max_{1\leq j\leq p}
    \abs{\sum_{i=1}^n\{U_i(j)-\Eb U_i(j)\}}
    &\leq
    \Eb\max_{1\leq j\leq p}\sum_{i=1}^nU_i(j)
    +
    \max_{1\leq j\leq p}\sum_{i=1}^n\Eb U_i(j)\\
    &\leq
    2\Eb\max_{1\leq j\leq p}\sum_{i=1}^nU_i(j).
\end{align*}
Combining this bound with \eqref{eq:bounded_part} proves \eqref{eq:k1_alpha_lt_1}.

It remains to consider $0<\alpha<1$.  By \eqref{eq:convex}, Lemmas~\ref{lem:bernoulli_moment} and \ref{lem:exp_power_sum}, and the same conditioning argument,
\begin{align*}
    &\left\|\sum_{i=1}^n\eta_i(A+\xi_i)^{1/\alpha}\right\|_{L_r}
    \\
    &\leq
    2^{1/\alpha-1}(e+1)
    \left[1+2^{2/\alpha-2}\{1+\Gamma(1+1/\alpha)\}\right]
    rA^{1/\alpha}
    \\
    &\quad+
    2^{2/\alpha-2}
    \left(1+\frac{1}{\alpha}\right)^{1/\alpha}
    \frac{1}{1-\alpha}r^{1/\alpha}.
\end{align*}
Multiplying by $p^{1/r}\leq e$, by $K(1+\log2)^{1/\alpha}$, and by the centering factor two gives the last two terms in \eqref{eq:k1_alpha_lt_1}.  Combining these with \eqref{eq:bounded_part} proves \eqref{eq:k1_alpha_lt_1}.
\end{proof}

\begin{proof}[Proof of Corollary~\ref{cor:mar_weibull_upper}]
Let $P^n\in\Pc_{n,p}^{\psi_\alpha,{\rm m}}(\sigma,K)$ and put $n^{-1}\sum_iX_i=S_n$.  Apply Lemma~\ref{lem:order} to $S_n$ and condition on the random subsets $J_{m_s}$, where $m_s=2^s\wedge k$ and $0\le s\le\ceil{\log_2 k}$.  For every realization of $J_{m_s}$, the selected subvector satisfies the same coordinate-wise variance and marginal $\psi_\alpha$ bounds in dimension $\ceil{p/m_s}$.  Theorem~\ref{thm:k1} therefore gives a bound for the conditional supremum norm.

The comparisons
\[
    \log(2\ceil{p/m_s})\le 2\log(2p/m_s),
    \qquad
    1\vee\log\ceil{p/m_s}\le C\log(2p/m_s),
\]
and
\[
    \log\left(e+\frac{n}{1\vee\log\ceil{p/m_s}}\right)
    \le
    C\log\left(e+\frac{n}{1\vee\log(2p/k)}\right)
\]
for a universal constant $C$ reduce the resulting sums to those in Lemma~\ref{lem:sum}, together with \eqref{eq:sum_log_power} for the last term when $0<\alpha<1$.  This proves \eqref{eq:cor:mar_weibull_upper} after increasing $C_\alpha$.
\end{proof}

\begin{proof}[proof of Theorem~\ref{thm:mar_weibull_lower_maximal}]

For the first bound, take
\begin{equation*}
    X_i(j)=(\sigma\wedge K)(\log2)^{1/\alpha}\varepsilon_{ij},
    \qquad 1\leq j\leq p,
\end{equation*}
where the $\varepsilon_{ij}$'s are independent Rademacher variables.  Then $\Eb X_i(j)=0$, $\Eb X_i(j)^2\leq\sigma^2$, and $\norm{X_i(j)}_{\psi_\alpha}\leq K$.  Combining Theorem 2.1 of \cite{chang2026notesconstantsmaximarademacher} with \eqref{eq:block_embedding_lower} proves \eqref{eq:lower_rademacher_sparse}.

For the second bound, set
\begin{equation*}
    \rho=\frac{1}{8np},
    \qquad
    A=K\{\log(1+\rho^{-1})\}^{1/\alpha}
    =K\{\log(1+8np)\}^{1/\alpha}.
\end{equation*}
Let $\eta_{ij}$ be independent Bernoulli$(\rho)$ variables, let $\xi_{ij}$ be independent Rademacher variables independent of the $\eta_{ij}$'s, and define $Z_i(j)=A\xi_{ij}\eta_{ij}.$ Then $\Eb Z_i(j)=0$ and
\begin{equation*}
    \Eb\exp\left[\left(\frac{|Z_i(j)|}{K}\right)^\alpha\right]
    =1-\rho+\rho\exp\{(A/K)^\alpha\}=2.
\end{equation*}
Also, \eqref{eq:bernoulli_rademacher_feasible} gives $\Eb Z_i(j)^2=A^2\rho\leq\sigma^2$.  The event that exactly one of the $n\ceil{p/k}$ Bernoulli variables $\eta_{ij}$ equals one has probability
\begin{equation*}
    n\ceil{p/k}\rho(1-\rho)^{n\ceil{p/k}-1}
    =
    \frac18\left(1-\frac{1}{8n\ceil{p/k}}\right)^{n\ceil{p/k}-1}
    \geq \frac{e^{-1/8}}{8},
\end{equation*}
where $\log(1-x)\geq -x/(1-x)$ is used with $x=1/(8n\ceil{p/k})$.  On this event,
\begin{equation*}
    \max_{1\leq j\leq \ceil{p/k}}
    \left|\frac{1}{n}\sum_{i=1}^nZ_i(j)\right|
    =\frac{A}{n}.
\end{equation*}
Therefore the expectation in dimension $\ceil{p/k}$ is at least $e^{-1/8}A/(8n)$, and \eqref{eq:bernoulli_rademacher_lower_sparse} follows from \eqref{eq:block_embedding_lower}.

It remains to prove the third bound.  First note that, for every $0<\rho\leq1/8$, if $Y$ is Bernoulli$(\rho)$ and
\begin{equation*}
    A=a_\alpha K\{\log(1/\rho)\}^{1/\alpha},
\end{equation*}
then $\norm{A(Y-\rho)}_{\psi_\alpha}\leq K$.  Indeed,
\begin{align*}
    \Eb\exp\left[\left(\frac{A|Y-\rho|}{K}\right)^\alpha\right]
    &=(1-\rho)\exp\{a_\alpha^\alpha\rho^\alpha\log(1/\rho)\}
      +\rho\exp\{a_\alpha^\alpha(1-\rho)^\alpha\log(1/\rho)\} \\
    &\leq
      \exp\left\{\frac{a_\alpha^\alpha}{e\alpha}\right\}
      +\rho^{1-a_\alpha^\alpha}
    \leq
      2.
\end{align*}
Here we used $\sup_{0<x<1}x^\alpha\log(1/x)=1/(e\alpha)$, $a_\alpha^\alpha\leq1/2$, and $\rho\leq1/8$.

Assume \eqref{eq:centered_bernoulli_feasible}.  Work in dimension $\ceil{p/k}$ and set
\begin{equation*}
    \rho=\frac{\log(2\ceil{p/k})}{8n},
    \qquad
    A=a_\alpha K\left\{\log\left(\frac{8n}{\log(2\ceil{p/k})}\right)\right\}^{1/\alpha}.
\end{equation*}
Let $Y_{ij}$ be independent Bernoulli$(\rho)$ variables and put $Z_i(j)=A(Y_{ij}-\rho)$.  Since $\log(2\ceil{p/k})\leq n$, $\rho\leq1/8$, and the preceding paragraph gives $\norm{Z_i(j)}_{\psi_\alpha}\leq K$.  The variance feasibility in \eqref{eq:centered_bernoulli_feasible} gives $\Eb Z_i(j)^2\leq\sigma^2$.

If $\log(2\ceil{p/k})\leq4$, then, for each coordinate,
\begin{align*}
    \Pb\left(\sum_{i=1}^nY_{ij}\geq1\right)
    &\geq n\rho(1-\rho)^{n-1}
     \geq {\frac{\log(2\ceil{p/k})}{8}}\exp\left\{-{\frac{\log(2\ceil{p/k})}{7}}\right\}.
\end{align*}
On this event,
\begin{equation*}
    \left|\frac1n\sum_{i=1}^nZ_i(j)\right|
    \geq {\frac{A}{n}}\left(1-{\frac{\log(2\ceil{p/k})}{8}}\right)
    \geq {\frac{A}{2n}}.
\end{equation*}
Since $\log(2\ceil{p/k})\geq\log2$, the expectation in dimension $\ceil{p/k}$ is at least
\begin{equation*}
    {\frac{e^{-4/7}\log2}{64}}\,
    A\frac{\log(2\ceil{p/k})}{n}.
\end{equation*}
If $4<\log(2\ceil{p/k})\leq8$, then, for each coordinate,
\begin{align*}
    \Pb\left(\sum_{i=1}^nY_{ij}\geq2\right)
    &\geq {\frac{n(n-1)}{2}}\rho^2(1-\rho)^{n-2}
      \geq {\frac{e^{-8/7}}{16}}.
\end{align*}
Here we used $n\geq\log(2\ceil{p/k})>4$, $\rho\leq1/8$, and $\log(1-x)\geq-x/(1-x)$.  On this event,
\begin{equation*}
    \left|\frac1n\sum_{i=1}^nZ_i(j)\right|
    \geq {\frac{A}{n}}\left(2-{\frac{\log(2\ceil{p/k})}{8}}\right)
    \geq {\frac{A}{n}}.
\end{equation*}
Since $\log(2\ceil{p/k})\leq8$, the expectation in dimension $\ceil{p/k}$ is at least
\begin{equation*}
    {\frac{e^{-8/7}}{128}}\,
    A\frac{\log(2\ceil{p/k})}{n}.
\end{equation*}
It remains to consider $\log(2\ceil{p/k})>8$.  For each coordinate,
\begin{align*}
    &\Pb\left(\sum_{i=1}^nY_{ij}=\floor{\log(2\ceil{p/k})/2}\right) \\
    &\geq
    \left(\frac{n}{\floor{\log(2\ceil{p/k})/2}}\right)^{\floor{\log(2\ceil{p/k})/2}}
    \left(\frac{\log(2\ceil{p/k})}{8n}\right)^{\floor{\log(2\ceil{p/k})/2}}
    (1-\rho)^n \\
    &\geq
    \exp\left[-\left\{\frac{\log4}{2}+\frac17\right\}\log(2\ceil{p/k})\right].
\end{align*}
Here we used $\binom{n}{m}\geq(n/m)^m$,
\[
    \floor{\log(2\ceil{p/k})/2}\leq \log(2\ceil{p/k})/2,
    \qquad
    (1-\rho)^n\geq\exp\{-\log(2\ceil{p/k})/7\}.
\]
Since
\[
    \ceil{p/k}=\exp\{\log(2\ceil{p/k})\}/2
    \quad\mbox{and}\quad
    1-(\log4)/2-1/7>0,
\]
we get
\begin{equation*}
    \Pb\left(
    \max_{1\leq j\leq \ceil{p/k}}\sum_{i=1}^nY_{ij}
    \geq\floor{\log(2\ceil{p/k})/2}
    \right)
    \geq
    \frac12.
\end{equation*}
On this event,
\begin{equation*}
    \max_{1\leq j\leq \ceil{p/k}}
    \left|\frac{1}{n}\sum_{i=1}^nZ_i(j)\right|
    \geq
    \frac{A}{n}
    \left(\floor{\frac{\log(2\ceil{p/k})}{2}}-\frac{\log(2\ceil{p/k})}{8}\right)
    \geq
    A\frac{\log(2\ceil{p/k})}{4n}.
\end{equation*}
Thus, in all three cases, the expectation in dimension $\ceil{p/k}$ is at least
\begin{equation*}
    {\frac{e^{-8/7}}{128}}\,
    A\frac{\log(2\ceil{p/k})}{n},
\end{equation*}
and \eqref{eq:centered_bernoulli_lower_sparse} follows from \eqref{eq:block_embedding_lower}.
    
\end{proof}

\begin{proof}[proof of Corollary~\ref{cor:mar_weibull_lower}]
The proof first treats $\ceil{p/k}$ coordinates and then embeds the construction in $\Real^p$.  Let $Z_i\in\Real^{\ceil{p/k}}$ be an iid construction satisfying the marginal variance and marginal $\psi_\alpha$ constraints with parameters $(\sigma,K)$.  Repeat each coordinate of $Z_i$ exactly $\floor{p/\ceil{p/k}}$ times and fill the remaining coordinates, if any, by zero.  Since
$\floor{p/\ceil{p/k}}\geq \frac{k}{2},$ we have
\begin{equation}\label{eq:block_embedding_lower}
    \left\|
    \frac{1}{n}\sum_{i=1}^nX_i
    \right\|_{(k),2}
    \geq
    \sqrt{\frac{k}{2}}
    \max_{1\leq j\leq \ceil{p/k}}
    \left|\frac{1}{n}\sum_{i=1}^nZ_i(j)\right|.
\end{equation} Now, Theorem~\ref{thm:mar_weibull_lower_maximal} completes the proof.
\end{proof}

\begin{proof}[proof of Corollary~\ref{cor:mar_weibull_characterization}]
The upper bound follows from Corollary~\ref{cor:mar_weibull_upper}. Since the
iid class is a subclass of the product class, the same upper bound also holds
for $\Ec_{n,p,k,\psi_\alpha}^{{\rm m},*}(\sigma,K)$.

It remains to prove the matching lower bound for
$\Ec_{n,p,k,\psi_\alpha}^{{\rm m},*}(\sigma,K)$. The first condition in
\eqref{eq:centered_bernoulli_feasible} implies
\[
    \min\left\{1,\sqrt{\frac{\log(2\ceil{p/k})}{n}}\right\}
    =
    \sqrt{\frac{\log(2\ceil{p/k})}{n}}.
\]
Therefore \eqref{eq:lower_rademacher_sparse}, together with
$\log(2p/k)\leq \log(2\ceil{p/k})$ and
$c\sigma\wedge C_\alpha K\lesssim_\alpha\sigma\wedge K$, gives the lower bound
matching the first summand in the right-hand side of
\eqref{eq:cor:mar_weibull_upper}.

Next, \eqref{eq:centered_bernoulli_feasible} and
\eqref{eq:centered_bernoulli_lower_sparse} give the lower bound matching the
second summand. Indeed, since $p/k\geq1$,
\[
    1\vee\log(p/k)
    \leq
    \frac{\log(2\ceil{p/k})}{\log2},
\]
and
\[
    \log(2\ceil{p/k})
    \leq
    3\{1\vee\log(p/k)\}.
\]
The second inequality follows by considering separately the cases
$1\leq p/k<e$ and $p/k\geq e$. Since
$\log(2\ceil{p/k})\leq n$, the preceding display implies
\[
    \log\left(e+\frac{n}{1\vee\log(p/k)}\right)
    \leq
    \log\left(\frac{8n}{\log(2\ceil{p/k})}\right).
\]
Consequently,
\[
\begin{aligned}
&\frac{1\vee\log(p/k)}{n}
    \left[
        \log\left(e+\frac{n}{1\vee\log(p/k)}\right)
    \right]^{1/\alpha}  \\
&\qquad\leq
    \frac{1}{\log2}\,
    \frac{\log(2\ceil{p/k})}{n}
    \left\{
        \log\left(\frac{8n}{\log(2\ceil{p/k})}\right)
    \right\}^{1/\alpha}.
\end{aligned}
\]
Thus \eqref{eq:centered_bernoulli_lower_sparse} controls the second summand in
the right-hand side of \eqref{eq:cor:mar_weibull_upper}, up to a constant
depending only on $\alpha$.

When $\alpha\in(0,1)$, \eqref{eq:bernoulli_rademacher_feasible} and
\eqref{eq:bernoulli_rademacher_lower_sparse} give the lower bound matching the
third summand. Indeed,
\[
    1\vee\log(p/k)
    \leq
    \log(1+8n\ceil{p/k}),
\]
and hence
\[
    \frac{\{1\vee\log(p/k)\}^{1/\alpha}}{n}
    \leq
    \frac{\{\log(1+8n\ceil{p/k})\}^{1/\alpha}}{n}.
\]
For $\alpha\geq1$, the third summand in
\eqref{eq:cor:mar_weibull_upper} is absent.

Combining the preceding lower bounds, we obtain that
$\Ec_{n,p,k,\psi_\alpha}^{{\rm m},*}(\sigma,K)$ is bounded below by a constant
depending only on $\alpha$ times each nonnegative summand in the right-hand
side of \eqref{eq:cor:mar_weibull_upper}. Since the maximum of finitely many
nonnegative numbers is at least their average, this gives the lower bound
matching the full right-hand side of \eqref{eq:cor:mar_weibull_upper}. The
corresponding lower bound for
$\Ec_{n,p,k,\psi_\alpha}^{{\rm m}}(\sigma,K)$ follows from
\[
    \Ec_{n,p,k,\psi_\alpha}^{{\rm m},*}(\sigma,K)
    \leq
    \Ec_{n,p,k,\psi_\alpha}^{{\rm m}}(\sigma,K).
\]
The proof is complete.
\end{proof}

\section{Useful Lemmas and Propositions}\label{app:aux-estimates}

For nonnegative random variables $U$ and $V$, write $U\preceq V$ if
\begin{equation*}
    \Pp\{U>t\}\leq \Pp\{V>t\},
    \qquad t\geq0.
\end{equation*}
Then $\Eb\phi(U)\leq \Eb\phi(V)$ for every nondecreasing measurable $\phi$ for which the expectations exist.  If $U_1,\ldots,U_n$ are independent, $V_1,\ldots,V_n$ are independent, and $U_i\preceq V_i$ for each $i$, then $\sum_i U_i\preceq \sum_iV_i$, by quantile coupling (see, for instance, Theorem 1.A.3 of \cite{Shaked2007}).

For $0<q\leq1$, $A\geq1$, and $x\geq0$,
\begin{equation}
    (A+x)^q\leq A^q+x^q\leq 2A^q+x .
    \label{eq:subadd}
\end{equation}
The first inequality is subadditivity of $x\mapsto x^q$.  For the second, if $x\leq A$, then $x^q\leq A^q$, while if $x>A$, then $x\geq1$ and $x^q\leq x$.  If $q>1$, then
\begin{equation}
    (A+x)^q\leq 2^{q-1}(A^q+x^q),
    \qquad A,x\geq0 .
    \label{eq:convex}
\end{equation}

\begin{lemma}[Bernstein integral]\label{lem:bernstein_integral}
Let $Y_i(j)$, $1\leq i\leq n$, $1\leq j\leq p$, be independent over $i$, centered, and satisfy
\begin{equation*}
    \abs{Y_i(j)}\leq b,
    \qquad
    \max_{1\leq j\leq p}\sum_{i=1}^n\Eb Y_i(j)^2\leq V .
\end{equation*}
Then
\begin{equation*}
    \Eb\max_{1\leq j\leq p}
    \abs{\sum_{i=1}^nY_i(j)}
    \leq
    \sqrt{2V}
    \left\{
    \sqrt{\log(2p)}+\frac{1}{2\sqrt{\log(2p)}}
    \right\}
    +\frac{2b}{3}\{\log(2p)+1\}.
\end{equation*}
\end{lemma}

\begin{proof}
Applying Bennett's inequality \citep{Bennett1962} (see Lemma~\ref{lem:bennett}) to both signs and taking a union bound gives, for all $x\geq0$,
\begin{equation*}
    \Pp\left\{
    \max_{1\leq j\leq p}\abs{\sum_{i=1}^nY_i(j)}
    >
    \sqrt{2V\{x+\log(2p)\}}+\frac{2b}{3}\{x+\log(2p)\}
    \right\}\leq e^{-x}.
\end{equation*}
Integrating this tail bound and using
\begin{equation*}
    \int_0^\infty \frac{e^{-x}}{\sqrt{x+\log(2p)}}\,dx
    \leq \frac{1}{\sqrt{\log(2p)}}
\end{equation*}
proves the result.
\end{proof}

\begin{lemma}\label{lem:bernoulli_moment}
Let $N$ be a sum of independent Bernoulli variables such that $\Eb N\leq r$ for some $r\geq1$. Then $\norm{N}_{L_r}\leq (e+1)r.$
\end{lemma}

\begin{proof}
Chernoff's inequality \citep{Chernoff1952} gives, for every $x\geq0$,
\begin{equation*}
    \Pp\{N\geq er+x\}
    \leq
    \left(\frac{er}{er+x}\right)^{er+x}
    \leq e^{-x},
\end{equation*}
where the last step follows from $(1+t)\log(1+t)\geq t$, $t\geq0$.  Hence $N\preceq er+\xi$, where $\xi$ is standard exponential.  Since $\norm{\xi}_{L_r}=\Gamma(r+1)^{1/r}\leq r$, the claim follows.
\end{proof}

\begin{lemma}\label{lem:gamma_moment}
If $\xi_1,\ldots,\xi_m$ are independent standard exponential variables, then, for all integers $m\geq0$ and all $r\geq1$,
\begin{equation*}
    \left\|\sum_{i=1}^m\xi_i\right\|_{L_r}\leq m+r .
\end{equation*}
\end{lemma}

\begin{proof}
The case $m=0$ is trivial.  For $m\geq1$, write $G=\sum_{i=1}^m\xi_i$.  For $0<\lambda<1$ and $y\geq0$,
\begin{equation*}
    y^r\leq \left(\frac{r}{e\lambda}\right)^r e^{\lambda y}.
\end{equation*}
Thus $\norm{G}_{L_r}\leq r(e\lambda)^{-1}(1-\lambda)^{-m/r}$.  Taking $\lambda=r/(m+r)$ gives
\begin{equation*}
    \norm{G}_{L_r}
    \leq \frac{m+r}{e}\left(1+\frac{r}{m}\right)^{m/r}
    \leq m+r .
\end{equation*}
\end{proof}

\begin{lemma}\label{lem:exp_power_sum}
Let $q>1$, and let $\xi_1,\ldots,\xi_m$ be independent standard exponential variables.  Then, for all integers $m\geq0$ and all $r\geq1$,
\begin{equation*}
    \left\|\sum_{i=1}^m\xi_i^q\right\|_{L_r}
    \leq
    2^{2q-2}\{1+\Gamma(q+1)\}m
    +2^{q-1}(q+1)^q\frac{q}{q-1}r^q .
\end{equation*}
\end{lemma}

\begin{proof}
The case $m=0$ is immediate.  Let $\xi_{(1)}\geq\cdots\geq\xi_{(m)}$ be the decreasing rearrangement.  For $1\leq s\leq m$ and $u\geq0$,
\begin{equation*}
    \Pp\left\{
    \xi_{(s)}>\log\left(\frac{em}{s}\right)+\frac{u}{s}
    \right\}
    \leq
    \binom{m}{s}\exp\left[-s\log\left(\frac{em}{s}\right)-u\right]
    \leq e^{-u}.
\end{equation*}
Hence $\xi_{(s)}\preceq \log(em/s)+\xi/s$, where $\xi$ is standard exponential.  By Lemma~\ref{lem:gamma_moment}, with moment order $qr$ and $m=1$,
\begin{equation*}
    \norm{\xi_{(s)}}_{L_{qr}}
    \leq
    \log\left(\frac{em}{s}\right)+\frac{qr+1}{s}
    \leq
    \log\left(\frac{em}{s}\right)+\frac{(q+1)r}{s}.
\end{equation*}
Minkowski's inequality in $L_r$ gives
\begin{align*}
    \left\|\sum_{i=1}^m\xi_i^q\right\|_{L_r}
    &\leq
    \sum_{s=1}^m\norm{\xi_{(s)}^q}_{L_r} \\
    &\leq
    2^{q-1}\sum_{s=1}^m\left\{\log\left(\frac{em}{s}\right)\right\}^q
    +2^{q-1}(q+1)^qr^q\sum_{s=1}^m s^{-q}.
\end{align*}
Furthermore,
\begin{equation*}
    \sum_{s=1}^m\left\{\log\left(\frac{em}{s}\right)\right\}^q
    \leq
    m\int_0^1\{\log(e/t)\}^q\,dt
    \leq
    2^{q-1}\{1+\Gamma(q+1)\}m,
\end{equation*}
and $\sum_{s=1}^m s^{-q}\leq q/(q-1)$.  Combining these estimates proves the lemma.
\end{proof}

\begin{lemma}[Bennett's inequality \citep{Bennett1962}]\label{lem:bennett}
Let $\xi_1,\ldots,\xi_n$ be independent mean-zero random variables satisfying $|\xi_i|\le K$ almost surely and $\sum_{i=1}^n\Eb\xi_i^2\le n\sigma^2$. Then, for every $x>0$,
\begin{equation}\label{eq:bennett_consequence}
    \Pb\left(\sum_{i=1}^n\xi_i
    \ge n\sigma\sqrt{2x/n}+Kx/3\right)
    \le e^{-x}.
\end{equation}
The same bound holds for the lower tail.
\end{lemma}

\begin{proof}
Bennett's inequality gives
\begin{equation*}
    \Pb\left(\sum_{i=1}^n\xi_i\ge nt\right)
    \le
    \exp\left[-{\frac{n\sigma^2}{K^2}}h\left({\frac{Kt}{\sigma^2}}\right)\right],
    \qquad
    h(y)=(1+y)\log(1+y)-y.
\end{equation*}
It is enough to prove $h(\sqrt{2s}+s/3)\ge s$ for every $s\ge0$. With $z=\sqrt{2s}$, this is $h(z+z^2/6)\ge z^2/2$. The derivative of the difference is
\begin{equation*}
    \left(1+{\frac{z}{3}}\right)\log(1+z+z^2/6)-z.
\end{equation*}
This derivative is nonnegative because
\begin{equation*}
    {\frac{d}{dz}}\left\{\log(1+z+z^2/6)-{\frac{z}{1+z/3}}\right\}
    ={\frac{z^2(2z+9)}{(z+3)^2(z^2+6z+6)}}\ge0,
\end{equation*}
and the expression in braces is zero at $z=0$. Substituting $s=K^2x/(n\sigma^2)$ proves \eqref{eq:bennett_consequence}.
\end{proof}

\begin{lemma}\label{lem:order}
Let $p\ge1$, $k\in[p]$, and $x\in\Real^p$. For $1\le m\le k$, let $J_m$ be uniformly distributed over subsets of $[p]$ of cardinality $\ceil{p/m}$. Let $|x|_{(m)}$ denote the $m$:th largest value among $|x(1)|,\ldots,|x(p)|$. Then
\begin{equation}\label{eq:order_hit}
    |x|_{(m)}\le (1-e^{-1})^{-1}\Eb_{J_m}\max_{j\in J_m}|x(j)|.
\end{equation}
Consequently, with $m_s=2^s\wedge k$ and $S=\ceil{\log_2 k}$,
\begin{equation}\label{eq:order_sparse}
    \norm{x}_{(k),2}
    \le
    (1-e^{-1})^{-1}\sum_{s=0}^S\sqrt{m_s}\,
    \Eb_{J_{m_s}}\max_{j\in J_{m_s}}|x(j)|.
\end{equation}
\end{lemma}

\begin{proof}
Let $T_m$ be the set of indices corresponding to the $m$ largest absolute coordinates. If $\ceil{p/m}\le p-m$, then
\begin{equation*}
    \Pb(J_m\cap T_m=\varnothing)
    ={\frac{\binom{p-m}{\ceil{p/m}}}{\binom{p}{\ceil{p/m}}}}
    \le\left(1-{\frac{m}{p}}\right)^{\ceil{p/m}}
    \le e^{-1}.
\end{equation*}
If $\ceil{p/m}>p-m$, the left side is zero. Therefore,
\begin{equation*}
    \Eb_{J_m}\max_{j\in J_m}|x(j)|\geq |x|_{(m)}\Pb(J_m\cap T_m\neq \varnothing)\geq (1-e^{-1})|x|_{(m)}.
\end{equation*} Thus \eqref{eq:order_hit} holds.
For nonincreasing $b_1\ge\cdots\ge b_p\ge0$,
\begin{equation*}
    \left(\sum_{j=1}^k b_j^2\right)^{1/2}
    \le\sum_{s=0}^S\sqrt{m_s}\,b_{m_s},
\end{equation*}
by partitioning $\{1,\ldots,k\}$ into the dyadic blocks generated by the $m_s$'s. Applying this with $b_j=|x|_{(j)}$ and using \eqref{eq:order_hit} proves \eqref{eq:order_sparse}.
\end{proof}

\begin{lemma}\label{lem:sum}
For $m_s=2^s\wedge k$ and $S=\ceil{\log_2 k}$, one has
\begin{equation}\label{eq:sum_one}
    \sum_{s=0}^S\sqrt{m_s}\le(3+\sqrt{2})\sqrt{k}.
\end{equation}
Moreover, for every $1/2\le\beta\le1$,
\begin{equation}\label{eq:sum_log}
    \sum_{s=0}^S\sqrt{m_s}
    \left({\frac{\log(2p/m_s)}{n}}\right)^\beta
    \le
    (9+5\sqrt{2})\sqrt{k}
    \left({\frac{\log(2p/k)}{n}}\right)^\beta.
\end{equation}
For every $\beta\ge0$ there is a finite constant $C_\beta$ such that
\begin{equation}\label{eq:sum_log_power}
    \sum_{s=0}^S\sqrt{m_s}
    \{1\vee \log^\beta(2p/m_s)\}
    \le C_\beta\sqrt{k}\{1\vee\log^\beta(2p/k)\}.
\end{equation}
If $\log\Lambda_q(k)>1$, then
\begin{equation}\label{eq:sum_log_delta}
    \sum_{s=0}^S\sqrt{m_s}
    \left(
    {\frac{\log(2p/m_s)}{n\log\left\{{\frac{B^2}{(\sigma\wedge B)^2}}
    \left({\frac{\log(2p/m_s)}{n}}\right)^{1-2/q}\right\}}}
    \right)^{1-1/q}\le
    (9+5\sqrt{2})\sqrt{k}
    \left({\frac{\log(2p/k)}{n\log\Lambda_q(k)}}\right)^{1-1/q}.
\end{equation}
\end{lemma}

\begin{proof}
For $s<S$, put $j=S-s$. Then $j\ge1$, $m_s=2^{S-j}$, $k>2^{S-1}$, and hence
\begin{equation*}
    \sqrt{m_s}\le \sqrt{2k}\,2^{-j/2}.
\end{equation*} This implies that $\sum_{s=0}^S\sqrt{m_s}=
    \sqrt{k}+\sum_{s=0}^{S-1}2^{s/2}\leq\sqrt{k}+\sqrt{2k}\sum_{j=1}^{\infty}2^{-j/2}=(3+\sqrt2)\sqrt{k}$. Also $k\le2^S$ gives
\begin{equation*}
    \log(2p/m_s)=\log(2p/k)+\log(k/m_s)\le \log(2p/k)+j\log2\le (j+1)\log(2p/k),
\end{equation*}
because $\log(2p/k)\ge\log2$. Therefore, for $1/2\le\beta\le1$,
\begin{align*}
    \sum_{s=0}^S\sqrt{m_s}
    \left({\frac{\log(2p/m_s)}{n}}\right)^\beta
    &\le \sqrt{k}\left({\frac{\log(2p/k)}{n}}\right)^\beta
    \left\{1+\sqrt2\sum_{j=1}^\infty 2^{-j/2}(j+1)\right\} \\
    &= (9+5\sqrt{2})\sqrt{k}\left({\frac{\log(2p/k)}{n}}\right)^\beta.
\end{align*} For the last equality, we used $1 + \sqrt{2}\sum_{j=1}^\infty (j+1)\left(\frac{1}{\sqrt{2}}\right)^j = 1 + \sqrt{2}\left[\frac{1}{(1-1/\sqrt{2})^2} - 1\right]=9+5\sqrt{2}$. The proof of \eqref{eq:sum_log_power} is identical: using
$\log(2p/m_s)\le (j+1)\log(2p/k)$ for $s<S$ gives a convergent geometric series
$\sum_{j\ge1}2^{-j/2}(j+1)^\beta$.

For \eqref{eq:sum_log_delta}, since $m_s\le k$ and $1-2/q\ge0$,
\begin{equation*}
\log\left\{
{\frac{B^2}{(\sigma\wedge B)^2}}
\left({\frac{\log(2p/m_s)}{n}}\right)^{1-2/q}
\right\}
\ge
\log\Lambda_q(k).
\end{equation*}
Hence
\begin{equation*}
\sqrt{m_s}
\left(
{\frac{\log(2p/m_s)}{n\log\left\{
{\frac{B^2}{(\sigma\wedge B)^2}}
\left({\frac{\log(2p/m_s)}{n}}\right)^{1-2/q}
\right\}}}
\right)^{1-1/q} \le
{\frac{\sqrt{m_s}\,{\log(2p/m_s)}^{1-1/q}}{n^{1-1/q}{\log\Lambda_q(k)}^{1-1/q}}} .
\end{equation*}
Summing over $s$ and applying the preceding bound with $\beta=1-1/q$ proves \eqref{eq:sum_log_delta}.
\end{proof}

\begin{lemma}\label{lem:block}
Let $p_k=\ceil{p/k}$. Then
\begin{equation}\label{eq:block_amp}
    \Ec^*_{n,p,k,q}(\sigma,B)
    \ge {\sqrt{\frac{k}{2}}}\Ec^*_{n,p_k,1,q}(\sigma,B)\quad\mbox{and}\quad \Ec_{n,p,k,\psi_\alpha}^{\infty,*}(\sigma,K)\geq\sqrt{\frac{k}{2}}\Ec_{n,p_k,1,\psi_\alpha}^{\infty,*}(\sigma,K).
\end{equation}If $\log\Lambda_q(k)>1$ and $\log(2p/k)\ge\log\Lambda_q(k)$, then
\begin{equation}\label{eq:condition_preserved}
    \log(2p_k)
    \ge
    \log\left\{ {\frac{B^2}{(\sigma\wedge B)^2}}
    \left({\frac{\log(2p_k)}{n}}\right)^{1-2/q}\right\},
\end{equation}
and
\begin{equation}\label{eq:scale_preserved}
    {\frac{\log(2p_k)}{\log\left\{ {\frac{B^2}{(\sigma\wedge B)^2}}
    \left({\frac{\log(2p_k)}{n}}\right)^{1-2/q}\right\}}}
    \ge
    {\frac{1}{2}}{\frac{\log(2p/k)}{\log\Lambda_q(k)}}.
\end{equation}
\end{lemma}

\begin{proof}
Note that $p/k\le p_k\le p/k+1$. For \eqref{eq:block_amp}, take an iid admissible law in dimension $p_k$. For $1\le i\le n$, write the corresponding vector as $X_i^0\in\Real^{p_k}$ and construct $X_i\in\Real^p$ by repeating $X_i^0$ exactly $\floor{p/p_k}$ times and filling the remaining coordinates, if any, with zeros; that is,
\begin{equation*}
X_i((r-1)p_k+\ell)=X_i^0(\ell),
\qquad
1\le r\le \floor{p/p_k},\quad 1\le \ell\le p_k,
\end{equation*}
and $X_i(j)=0$ for the remaining coordinates. The variance and envelope constraints are unchanged, because every nonzero coordinate of $X_i$ is one coordinate of $X_i^0$. Since $\floor{p/p_k}\geq \ceil{k/2}$, for every realization,
\begin{equation*}
\left\Vert {\frac{1}{n}}\sum_{i=1}^nX_i\right\Vert_{(k),2} \ge
\sqrt{\floor{p/p_k}\wedge k}
\left\Vert {\frac{1}{n}}\sum_{i=1}^nX_i^0\right\Vert_\infty  \ge
{\sqrt{\frac{k}{2}}}
\left\Vert {\frac{1}{n}}\sum_{i=1}^nX_i^0\right\Vert_\infty .
\end{equation*}
Taking expectations and then taking the supremum over all admissible laws in dimension $p_k$ proves \eqref{eq:block_amp}.

It remains to prove \eqref{eq:condition_preserved} and \eqref{eq:scale_preserved}. We have
\begin{equation*}
\log\left\{ {\frac{B^2}{(\sigma\wedge B)^2}}
    \left({\frac{\log(2p_k)}{n}}\right)^{1-2/q}\right\} =
    \log\Lambda_q(k)
    +\left(1-{\frac{2}{q}}\right)
    \log\left({\frac{\log(2p_k)}{\log(2p/k)}}\right).
\end{equation*}
The second term is between zero and $\log2$. Since $\log\Lambda_q(k)>1$, the denominator in \eqref{eq:scale_preserved} is at most $2\log\Lambda_q(k)$, while the numerator is at least $\log(2p/k)$. This proves \eqref{eq:scale_preserved}. For \eqref{eq:condition_preserved}, use $\log x\le x-1$ for $x\ge1$ to obtain
\begin{align*}
&\log\Lambda_q(k)+\left(1-{\frac{2}{q}}\right)\log\left({\frac{\log(2p_k)}{\log(2p/k)}}\right)   \le
    \log(2p/k)+\log\left({\frac{\log(2p_k)}{\log(2p/k)}}\right)  \\
&\qquad\le
    \log(2p/k)+{\frac{\log(2p_k)}{\log(2p/k)}}-1
    \le \log(2p_k),
\end{align*}
where the last inequality is equivalent to
\begin{equation*}
    \{\log(2p/k)-1\}
    \left({\frac{\log(2p_k)}{\log(2p/k)}}-1\right)\ge0.
\end{equation*}
This proves the lemma.
\end{proof}

\bibliographystyle{plainnat}
\bibliography{bib}
\end{document}